\documentclass{amsart}
\usepackage{amsthm}
\usepackage{hyperref}
\usepackage{graphicx}

\newtheorem{theorem}{Theorem}[section]
\newtheorem{definition}[theorem]{Definition}
\newtheorem{lemma}[theorem]{Lemma}
\newtheorem{proposition}[theorem]{Proposition}
\newtheorem{assumption}[theorem]{Assumption}
\newtheorem{corollary}[theorem]{Corollary}
\newtheorem{remark}[theorem]{Remark}
\newtheorem{example}[theorem]{Example}
\newcommand{\dom}{\mathrm{dom}}

\renewcommand{\>}{\rangle}
\newcommand{\E}{\mathbb{E}}     
\newcommand{\CC}{\mathcal C}
\newcommand{\GG}{\mathcal G}

\usepackage{color}
\definecolor{darkgreen}{rgb}{0, .5, 0}
\definecolor{darkred}{rgb}{.5, 0, 0}

\renewcommand{\deg}{\mathrm{deg}}  

\title{Abstract polynomial processes}
\date{\today}
\thanks{We are grateful to Julia Eisenberg for discussions.}
\author{Fred Espen Benth, Nils Detering and Paul Kr\"uhner}
\address{Fred Espen Benth \\
University of Oslo\\
Department of Mathematics \\
P.O. Box 1053, Blindern\\
N--0316 Oslo, Norway}
\email[]{fredb\@@math.uio.no}
\address{Nils Detering \\ 
University of California at Santa Barbara\\
Department of Statistics and Applied Probability\\
CA 93106 Santa Barbara, USA}
\email[]{detering\@@pstat.ucsb.edu}
\address{Paul Kr\"uhner \\
WU Vienna \\
Institute for Statistics and Mathematics \\
Welthandelsplatz 1
1020 Vienna, Austria
}
\email[]{paul.eisenberg\@@wu.ac.at}


\begin{document}
\maketitle


\begin{abstract}
We suggest a novel approach to polynomial processes solely based on a polynomial action operator.
With this approach, we can analyse such processes on general state spaces, going far beyond Banach spaces. Moreover, we can be very flexible in the definition of what 'polynomial' means. We show that 'polynomial process' universally means 'affine drift'. Simple assumptions on the polynomial action operators lead to stronger characterisations on the polynomial class of processes. 

In our framework we do not need to specify polynomials explicitly but can work with a general sequence of graded vector spaces of functions on the state space. Elements of these graded vector spaces form the monomials by introducing a sequence of vector space complements. The basic tool of our analysis is the polynomial action operator, which is a semigroup of operators mapping conditional expected values of monomials acting on the polynomial process to monomials of the same or lower grade. Unlike the classical Euclidean case, the polynomial action operator may not form a finite-dimensional subspace after a finite iteration, a property we call locally finite. 
We study abstract polynomial processes under both algebraic and topological assumptions on the polynomial actions, and establish an affine drift structure. Moreover, we characterize the covariance structure under similar but slightly stronger conditions. A crucial part in our analysis is the use of the (algebraic or topological) dual of the monomials of grade one, which serves as a linearization of the state space of the polynomial process. Our general framework covers polynomial processes with values in Banach spaces recently studied by Cuchiero and Svaluto-Ferro \cite{CS-F}.
\end{abstract}


\section{Introduction}
A polynomial process $X$ is characterized by the fact that for any polynomial $p$ of degree $m$ there exists another polynomial $q$ of degree at most $m$ such that $\E [p(X(s+h))|\mathcal{F}_s]=q(X_s)$, $s,h\geq 0$. Polynomial processes have been studied on different state spaces as for instance $\mathbb{R}^d$ or subsets thereof. At the beginning of a study of polynomial processes stands the specification of a space of functions which are considered polynomials. For the real-valued setting for instance, the polynomials up to degree $m$ are defined by 
$$
\mathcal{P}_m := \{x\mapsto  \sum_{k=0}^m a_k x^k \}
$$
and form an $m$-dimensional vector space and the space of all polynomials is $\mathcal{P}:=\cup_{m=0}^{\infty} \mathcal{P}_m $. It is equipped with a {\em degree function} $\text{deg} :\mathcal{P} \rightarrow \mathbb{N}$ that naturally assigns to each polynomial its degree. 

Mimicking such a nested vector space structure, in this paper, we start abstractly with a set of states $E$ (state space) and a sub-vector space $\mathcal{P}$ of the set of functions from $E$ to a field $\mathbb{F}$ together with a degree function $\text{deg} :\mathcal{P} \rightarrow \mathbb{N}$. This degree function then defines a {\em gradation sequence} $\mathcal{P}_1 \subset \mathcal{P}_2 \dots $ of sub-vector spaces $\mathcal P_n := \{p\in\mathcal P: \deg(p) \leq n\}$     and a decomposition into a direct sum $\mathcal{P}=\bigoplus_{k=0}^{\infty} \mathcal{M}_n$ where $\mathcal{M}_n$ is the vector space complement of $\mathcal{P}_{n-1}$ in $\mathcal{P}_{n}$, and contains the {\em monomials} of degree $m$. We then call an 
$E$-valued process $X(t)_{t\geq 0}$ polynomial if for each function $p \in \mathcal{P}_n$ there exists a function $q\in \mathcal{P}_n$ such that $\E [p(X(s+h))|\mathcal{F}_s]=q(X_s)$, defining a family of semigroups which we call polynomial action operator of the polynomial process. 

Our abstract setting allows several major advantages compared to the classical theory: By encoding the polynomial property in the polynomial action operator, we can analyse the processes on general state spaces, going far beyond Banach state spaces. Indeed, set $E$ does not need to be a subset of a topological space but can just be any set to start with. Moreover, we can be very flexible in the definition of what 'polynomial' means, as we can include functions which are usually not considered polynomials but still generate an invariant class under conditional expectation.  
In this abstract setting, we are able to show that 'polynomial process' universally means 'affine drift'. Simple additional assumptions on the polynomial action operators lead to stronger characterisations on the polynomial class of processes, including an understanding of the quadratic covariation.


{\em Related literature:}
Affine processes, which play a prominent role in finance and economics, have first been analysed systematically in
Duffie, Filipovi{\'{c}} and Schachermayer \cite{duffie.et.al.03}. Polynomial processes offer a generalization of affine processes --- short the additional requirement that they assume finite absolute power moments of any order --- and were first introduced in Cuchiero  \cite{cuchiero.dissertation}. Since then, polynomial processes received a lot of attention and they have been studied on different state spaces including $\mathbb{R}^n$ and subsets thereof (see Filipovi{\'{c}} and Larsson \cite{FL,FL2}, for instance). They have also found many applications in financial and insurance mathematics (see Ackerer, Filipovi{\'c} and Pulido \cite{AFP}, Benth and Lavagnini \cite{BL}, Biagini and Zhang \cite{Biag-Zhang}, Cuchiero, Keller-Ressel and Teichmann \cite{CKT}, Filipovi{\'{c}} \cite{F}, Filipovi{\'{c}}, Larsson and Pulido \cite{FLP},  
Kleisinger-Yu {\it et al.} \cite{K2019MF} and Ware \cite{F2017P}, for instance).
We also would like to mention the linear rational term structure models proposed in Filipovi{\'{c}}, Larsson and Trolle \cite{FLT-JF17}. Here, the diffusion model has an affine structure of the drift, something we recover under rather mild conditions for our general polynomial processes.

Recently, an infinite-dimensional extension of polynomial processes has been proposed in Cuchiero and Svaluto-Ferro \cite{CS-F} and Benth, Detering and Kr\"uhner \cite{BDK-polybanach}, where, in the latter article, multi-linear maps have been used as the replacement for classical monomials. In particular, in an infinite-dimensional setting some flexibility arises when specifying the class of polynomials. As each of these specifications gives rise to a nested vector space structure considered here, our general analysis is particularly valuable in an infinite-dimensional setting.  

{\em Main results and discussion on Cuchiero and Svaluto-Ferro's paper 'Infinite dimensional polynomial processes' \cite{CS-F}:}
The paper \cite{CS-F} by Cuchiero and Svaluto-Ferro has overlap with our paper and we like to discuss their results and differences to the present paper from our perspective. Cuchiero and Svaluto-Ferro \cite{CS-F} consider as space of polynomials the algebra generated by a subset of continuous linear functionals on a Banach space, while we understand by polynomials just any given graded vector space of functions. Recall that classical polynomials as well as the algebra generated by a subset of linear functions has a natural inherent graded structure which is compatible with pointwise multiplication, thus the polynomials considered in \cite{CS-F} are naturally a graded algebra. 

Further, in \cite{CS-F} polynomial processes with values in a Banach space are described from the martingale problem perspective, i.e.\ a polynomial process is described via the action of its generator on polynomials. We do not assume any linear structure for the state space of the polynomial process and start from the above mentioned invariance condition for conditional expectations, i.e.\ we assume that conditional expectations of polynomials are polynomials again. Cuchiero and Svaluto-Ferro \cite{CS-F} prove that under some technical condition this invariance also holds in their setup while we give natural examples in our setup where the generator description breaks down, cf.\ Example \ref{ex:outside drift} where the constant part of the drift points out of the given Hilbert space while other objects are as regular as one would expect in a stochastic partial differential equation-description for polynomial processes. Allowing polynomial processes to live on non-linear structures has the advantage when studying them, for instance, on Lie-groups where they have a natural interpretation as well --- recall that L\'evy processes have been studied on Lie groups, cf.\ Ming \cite{Ming}.

The main results of Cuchiero and Svaluto-Ferro \cite[Theorems 3.4 and 3.8]{CS-F} are the dual and bidual moment formulas. The dual moment formula describes the time-evolution of $\E[p(X_t)]$ in terms of the generator where $p$ is a polynomial and $X$ a polynomial process. We recover this formula in our setting under some algebraic (Proposition \ref{p:Matrix exponential}) and analytic (Lemma \ref{l:Mp martingale}) conditions. An advantage of these formulas in our setup is that polynomials up to order $m$ can be reinterpreted as polynomials of order $1$ which allows to apply formulas for first order polynomials to higher order polynomials. The bidual moment formula in Cuchiero and Svaluto-Ferro \cite{CS-F} describes the time evolution of $\E[(X_t^{\otimes n})_{n=1,\dots,k}]$, that is, of the expectation of the $k$-th tensor-power of $X$, which is in some situations easier to handle than the specific moments $\E[p(X_t)]$ for a given polynomial. In many cases, $\E[p(X_t)]$ can be recovered from the expectation of the tensor powers. We divert from this approach and provide instead a drift-martingale-type decomposition of a polynomial process, cf.\ Theorem \ref{t:affine drift algebraic} and Theorem \ref{t:affine drift}, which allows to recover $\E[p(X_t)]$ directly by applying expectations. Again, by reinterpreting higher order polynomials as first order polynomials and reapplying our results we recover the expectations of the higher tensor powers in the case when we mean by polynomials the algebra generated by linear functionals. If we additionally assume a multiplicative structure on the polynomials, then we can additionally recover the quadratic covariation of the given process, cf.\ Theorem \ref{t:quadratic covariance} and Theorem \ref{t:L2-structure}.

\subsection{Notations}
  We denote by $\mathbb F$ either the field of real numbers $\mathbb R$ or complex numbers $\mathbb C$. $\mathbb N$ resp.\ $\mathbb R_+$ denote the set of integers resp.\ real numbers which are greater or equal zero. $(\Omega,\mathfrak A,(\mathcal F_t)_{t\geq 0},P)$ denotes a filtered probability space. For an $\mathbb F$-valued function $f$ we denote by $\overline f$ its complex conjugate.
  We will use the letter $I$ for the identity operator on any $\mathbb F$-vector space $V$. If $V$ is a vector space over $\mathbb F$, then by a semigroup (resp.\ group) of operators we mean a family $(T_h)_{h\geq 0}$ (resp.\ $(T_h)_{h\in\mathbb F}$) of linear operators on $V$ with $T_0=I$ is the identity on $V$ and $T_hT_k=T_{h+k}$ for any $h,k\geq 0$ (resp.\ $h,k\in\mathbb F$). For a set $A$ we denote by $1_A$ the functional class which maps any element inside $A$ to $1$ and any other element to $0$.

\section{Polynomials, Polynomial Actions and Polynomial Processes}\label{sec:multilin:process new}
In this section we introduce abstract polynomial processes with values on a general non-empty set $E$. We first recall the classical $\mathbb R^d$-valued idea behind polynomial processes (see Cuchiero \cite{cuchiero.dissertation}).
  From this idea a polynomial process $X$ is a c\`adl\`ag process such that for any 'polynomial' $p$ and any $h\geq 0$ there is a polynomial $q$ with degree less or equal than the degree of $p$ such that for any $t\geq 0$
   $$ \E[ p(X_{t+h}) |\mathcal F_t] = q(X_t). $$
  For this, we need to define what we mean by 'polynomials'. This does, of course, induce an action $T_h$ on the set of polynomials which with $h$ fixed produces a polynomial $q=T_hp$ out of a polynomial $p$. For this to work well with conditional expectations, this action should be
   \begin{enumerate}
      \item linear,
      \item have the semigroup property $T_hT_k=T_{h+k}$ for $h,k\geq 0$ and
      \item $h\mapsto T_hp$ be right-continuous.
   \end{enumerate}
   These properties are motivated by the linearity, the tower property and continuity-properties of the conditional expectation, respectively. 

In order to make sense out of the above remark in an abstract setting with states of the process $X$ in a general set $E$, we need to:
\begin{itemize}
   \item[a)] Introduce what we mean by polynomials.
   \item[b)] Define operators $T_h$, $h\geq 0$ on the polynomials with $T_hT_k = T_{h+k}$ for $h,k\geq 0$.
   \item[c)] Say what we mean by $E$-valued polynomial processes.
\end{itemize}

We continue with formalising our ideas. In the rest of the paper, we reserve the notation $E$ for the state space of our polynomial processes to be defined, and $\mathcal F(E,\mathbb F)$ is the set of functions from $E$ to $\mathbb F$.  
\begin{definition}
\label{def:polynomial}
  By {\em polynomials} we mean a sub-vector space $\mathcal P\subseteq \mathcal F(E,\mathbb F)$ containing all constant functions, together with a degree-function 
   $$ \deg :\mathcal P\rightarrow \mathbb N $$
   such that $\deg(p+q) \leq \max\{\deg(p),\deg(q)\}$ and $\deg(\lambda p)=\deg(p)$ for any $p,q\in\mathcal P$, $\lambda \in\mathbb F\backslash\{0\}$ and $\deg(p) = 0$ if and only if $p$ is a constant function.
   
\end{definition}
Note that we use the convention $\deg(0) = 0$ in the definition above.  Allowing for $\mathbb C$-valued polynomials rather than $\mathbb R$-valued polynomials does not extend the theory but in some examples it may simplify matters.
   
We introduce the subset of polynomials of {\em order $n$} as
\begin{equation}
  \mathcal P_n := \{p\in\mathcal P: \deg(p) \leq n\}    
\end{equation}
and we fix a vector space complement $\mathcal L$ of the constant functions $\mathcal P_0$ in $\mathcal P_1$ which we call {\em linear functions}. We note that this set can be viewed as a realisation of monomials of order 1, and will play a crucial role in the discussions to come regarding the drift of polynomial processes. We state an example being important for our analysis. 
\begin{example}
Fix $x_0\in E$. Define 
   $$ \mathcal L:= \{p\in\mathcal P_1: p(x_0) = 0\}. $$
  Then $\mathcal L$ is a vector space complement of $\mathcal P_0$ in $\mathcal P_1$. This shows that these vector space complements always exist. However, sometimes it is more convenient to work with a different set $\mathcal L$, as the one we state here. To illustrate this point let $E=[1,2]^d$ and $\mathcal P$ be the set of classical $d$-variable polynomials with values in $\mathbb R$ restricted to $E$. The linear functions $\mathcal L$ on $\mathbb R^d$ restricted to $E$ is a natural choice of complement to the constant polynomials but it is not given via the above construction.
\end{example}
   
Next, we define {\em polynomial actions}: \begin{definition}
\label{def:poly-action}
  A {\em polynomial action} is a family $(T_h)_{h\geq 0}$ of linear operators on $\mathcal P$ with 
  \begin{enumerate}
     \item $\deg(T_hp)\leq\deg(p)$ for any $p\in\mathcal P$, $h\geq 0$,
     \item $T_0p = p$ for any $p\in\mathcal P$,
     \item $T_hT_k=T_{h+k}$ for any $h,k\geq 0$ and
     \item $\mathbb R_+\rightarrow\mathbb C, h\mapsto (T_hp)(x)$ is right-continuous for any $p\in\mathcal P$, $x\in E$.
\end{enumerate}    
    The {\em generator} $\mathcal G$ of $T$ is defined as
  $$ \mathcal Gp: \mathcal D \rightarrow\mathbb R, x\mapsto \lim_{h\searrow 0} \frac{T_hp(x)-p(x)}{h}$$
  where
  $$\mathcal D := \left\{p\in\mathcal P: \exists q\in\mathcal P\, \text{s.t.}\, \lim_{h\searrow0}\frac{T_hp(x)-p(x)}{h} =q(x)\,\,\forall x\in E\right\}$$
  and we call $\mathcal D$ the {\em domain of $\mathcal G$}.
\end{definition}

Finally, we are ready to introduce what we shall mean by {\it polynomial processes} in this paper:  
\begin{definition}
\label{def:poly-process}  
  A function $X:\Omega\times \mathbb R_+ \rightarrow E$ is called an {\em $E$-valued $\mathcal P$-polynomial process with action $T$} if for any $p\in\mathcal P$:
  \begin{enumerate}
     \item $\E[|p(X_t)|] < \infty$, $t\geq 0$,
     \item $\E[p(X_{t+h})| \mathcal F_t] = T_hp(X_t)$, $t,h\geq 0$,
     \item $\mathbb R_+\rightarrow\mathbb F, t\mapsto p(X_t)$ has c\`adl\`ag paths and
     \item $p(X_t)$ is $\mathcal F_t$-measurable for any $t\geq0$.
  \end{enumerate}
 If there is no ambiguity, then we sometimes simply refer to $X$ as a polynomial process.
\end{definition}
The c\`adl\`ag-property (3) in the definition of polynomial processes above implies that $p(X)$ is c\`adl\`ag for any $p\in\mathcal L$ and, hence, we may say that $X$ is weakly c\`adl\`ag. 
\begin{remark}
\label{rem:T_id_on_const}
Notice that from the degree reducing property (1) of polynomial actions in Definition \ref{def:poly-action}, it follows that $T_hp\in\mathcal P_0$ whenever $p\in\mathcal P_0$. I.e., from property (2) in Definition \ref{def:poly-process} linking the polynomial process to the action $T$, we see that $T$ is the identity operator on the constant functions, $T_hp=p, \forall p\in\mathcal P_0$. Hence, the domain $\mathcal D \supseteq \mathcal P_0$ contains the constant functions and the generator $\mathcal G$ maps constant functions to $0$.
\end{remark}

Sometimes it is useful to disregard higher order polynomials. This allows to omit definitions like $m$-polynomial processes which appears for instance in Cuchiero \cite{cuchiero.dissertation}.
\begin{remark}
  Let $n\in \mathbb N$. Then $(\mathcal P_n,\deg)$ is a graded vector space (in the sense of Bourbaki, N. (1974) Algebra I. Chapter 2, Section 11.) and $X$ is a $\mathcal P_n$-polynomial process with action $(T_h|_{\mathcal P_n})_{h\geq 0}$. 
\end{remark}

We end this section with a number of examples of polynomial processes as we have defined them, with a link to existing theory on polynomial processes.

\begin{example}
Classical polynomial processes in the sense of Cuchiero \cite{cuchiero.dissertation} are of course polynomial in the sense of Definition \ref{def:poly-process}. Indeed, let $E\subseteq \mathbb R^d$ be closed and $\mathcal P$ be the set of classical $d$-variable polynomials restricted to $E$, with the degree function defined as the order of the polynomials.

\end{example}

The next example shows that by taking the 'right' set of polynomials many Markov processes can be understood as polynomial processes. Indeed, if $\mathcal P$ is the vector space generated by eigenfunctions of the generator, then the Markov process is a $\mathcal P$-polynomial process. 
\begin{example}\label{ex: Levy}
  Let $L$ be an $\mathcal F$-L\'evy process on $E = \mathbb R^d$ and $\mathcal P$ the vector space generated by the functions $E\rightarrow \mathbb C, x\mapsto e^{iux}$ where $u\in\mathbb R^d$ and $ux := \sum_{j=1}^d \overline x_j u_j$ denotes the standard scalar product. We define 
  $$\deg:\mathcal P\rightarrow \{0,1\}, p\mapsto 1_{\{p\text{ non-constant}\}}.$$
  Observe that $\mathcal P=\mathcal P_1$. Take  $\mathcal L:= \{p\in\mathcal P: p(0) = 0\}$ to be the linear functions
  (or, "first order monomials"). Let $\psi$ be the L\'evy exponent of $L$, i.e.\ $\psi:\mathbb R^d\rightarrow \mathbb C$ is such that $\E[ e^{iuL(t)} ] = \exp(t\psi(u))$ for any $t\geq 0$, $u\in\mathbb R^d$ and let 
   $$ T_h:\mathcal P\rightarrow \mathcal P, \sum_{n=1}^k c_n e^{iu_n\cdot} \mapsto \sum_{n=1}^k c_n \exp(h\psi(u_n)) e^{iu_n\cdot}.$$
   Then $L$ is an $E$-valued $\mathcal P$-polynomial process with action $T$.
   
   An other possible choice of polynomials is the set of real-valued continuous functions on the one-point compactification of $E$.
\end{example}

Affine processes are thought to be polynomial processes if they satisfy certain moment conditions \cite[Example 4.4.1]{cuchiero.dissertation}. While this is true, they are always polynomial processes when polynomials are interpreted in a different way.
\begin{example}
  Let $X$ be an $\mathcal F$-affine process with values on $\mathbb R^d$, i.e.\ $X$ is a c\`adl\`ag process and there are functions $\psi:\mathbb R_+\times \mathbb C^d\rightarrow\mathbb C^d$ and $\phi:\mathbb R_+\times \mathbb C^d\rightarrow \mathbb C$ such that 
   $$ \E[ e^{iuX_{t+h}} |\mathcal F_t] = \exp(\phi(h,u) + \psi(h,u) X_t) $$
   and such that the flow property
    $$ \phi(h+k,u) = \phi(h,u)+\phi(k,\psi(h,u)),\quad \psi(h+k,u) = \psi(k,\psi(h,u))$$
   holds for any $t,h,k\geq 0$, $u\in\mathbb R^d$ 
   
  We define $(\mathcal P,\deg)$ as in Example \ref{ex: Levy} and introduce
   $$ T_h : \mathcal P\rightarrow\mathcal P, \sum_{n=1}^k c_n e^{iu_n\cdot} \mapsto \sum_{n=1}^k c_n \exp(\phi(h,u) + \psi(h,u)(\cdot)).$$
  
  Then $X$ is an $E$-valued $\mathcal P$-polynomial process with action $T$. For more details on affine processes we refer to Duffie et al. \cite{duffie.et.al.03}.
\end{example}

The next example refers back to a recent study of Cuchiero, Larsson and Svaluto-Ferro \cite{sulito.cuchiero.larsson}.

\begin{example}
   Let $(S,d)$ be a complete metric space, $E$ the set of Borel probability measures on $S$, and $\mathcal L$ the set of functions of the form
    $$ 
    E \rightarrow\mathbb R, \mu\mapsto \int f d\mu. 
    $$
    Here, $f:S\rightarrow \mathbb R$ is a continuous function vanishing at infinity, i.e.\ it has a continuous extension to the one-point compactification of $S$ with $f(\Delta)=0$ where $\Delta$ is the additional point. Further, let $\mathcal P$ be the algebra generated by $\mathcal L$ and the constant functions and $\deg:\mathcal P\rightarrow \mathbb N$ with
     $$ \deg\left(c_0 + \sum_{i=1}^n c_i \prod_{j=1}^{k_i} f_{i,j}\right) = \begin{cases} 0 & \text{if }n=0,\\ \max\{k_i:i=1,\dots,n\} & \text{otherwise.}\end{cases}$$
     where $n\in\mathbb N$, $k_1,\dots,k_n\in\mathbb N$, $c_1,\dots,c_n\in\mathbb R$, $f_{i,j}\in\mathcal L$. $\mathcal P$-polynomial processes on this structure have been investigated in Cuchiero, Larsson and Svaluto-Ferro \cite{sulito.cuchiero.larsson}.
\end{example}

\section{Algebraic Structure of Locally Finite Polynomial Actions}
We start our analysis of abstract polynomial process $X$ or rather their corresponding actions defined in Definitions \ref{def:poly-process}, \ref{def:poly-action} based on purely algebraic considerations. That is, we will not make any topological assumptions on $E$, the space $\mathcal P$, or even $\mathcal L$ in this Section. We recall that $X$ is an $E$-valued $\mathcal P$-polynomial process according to Definition \ref{def:poly-process} for a given set of polynomials $(\mathcal P,\deg)$ (Definition \ref{def:polynomial}) and $T$ being a polynomial action (Definition \ref{def:poly-action}). 

The property $\deg(T_hp)\leq\deg(p)$ of a polynomial action $T$ on $p\in\mathcal P$ means that $\mathcal P_n$ is $T$-invariant, that is, $T$ respects the gradation $\deg$ on $\mathcal P$. This has the immediate implication that $T$ can be understood as a 'triangular' type (matrix) operator on $\mathcal P_n$, where the diagonal elements are operators on vector space complements of $\mathcal P_k$ in $\mathcal P_{k+1}$.
To make this precise, we need first to introduce what we shall understand as higher-order monomials in abstract polynomials $\mathcal P$.

For this, let $\mathcal M_n$ be a vector space complement of $\mathcal P_{n-1}$ in $\mathcal P_{n}$, i.e. $\mathcal M_n$
\begin{equation}\label{def:monomials}
    \mathcal P_{n}=\mathcal P_{n-1}\oplus \mathcal M_n
\end{equation}
for $n\geq 2$, $\mathcal M_1:=\mathcal L$ and $\mathcal M_0:=\mathcal P_0$ be the set of constant functions on $E$. Further define the linear projections
\begin{equation}\label{def:projections}
  \Pi_n:\mathcal P\rightarrow \mathcal M_n  
\end{equation}
with kernel $\mathcal P_{n-1}\oplus \sum_{k=n+1}^\infty \mathcal M_k$ for $n\geq 0$ where $\mathcal P_{-1} := \{0\}$. The first order monomials $\mathcal M_1$ will play a prominent role in the construction of the linearization of $E$ and we thus keep the special notation $\mathcal L$. 

The next lemma shows that $T$ has an upper-block-triangle form relative to the algebraic decomposition $\mathcal P=\mathcal M_0\oplus\mathcal M_1\oplus\dots$.
\begin{lemma}\label{l:T triangle}
    It holds that 
    $$ T^{(n)}_h := \Pi_n T_h|_{\mathcal M_n} $$
  is a semigroup of linear operators on $\mathcal M_n$ and
   $$ T_h = \sum_{n=0}^\infty \sum_{l=0}^n \Pi_l T_h \Pi_n = \sum_{n=0}^\infty\left( T^{(n)}_h\Pi_n + \sum_{l=0}^{n-1} \Pi_l T_h \Pi_n\right)$$
   where, when inserting a polynomial, at most finitely many summands are non-zero.
\end{lemma}
\begin{proof}
  Let $h\geq 0$. For $p\in\mathcal P$ we find with $k:=\deg(p)$ that $p\in\mathcal P_k$ and, hence, $\Pi_n p = 0$ for any $n>k$. Thus, we have $p = \sum_{n=0}^\infty \Pi_np$ because there are at most finitely many non-zero summands. We find
  $$ T_h = \sum_{n=0}^\infty \sum_{l=0}^\infty \Pi_lT_h\Pi_n $$
   and we have $\Pi_lT_h\Pi_n = 0$ if $l>n$, $l,n\in\mathbb N$ because $T_h\Pi_n\mathcal P = 
  T_h\mathcal M_n \subseteq \mathcal P_n$ and $\Pi_l\mathcal P_n = \{0\}$. The claim follows.
\end{proof}
Recall from Remark \ref{rem:T_id_on_const} that $T_h$ acts as the identity operator on constants. 
Rewriting the last equality in Lemma \ref{l:T triangle} as a matrix operator on $\mathcal M_0\oplus \mathcal M_1\oplus \mathcal M_2\oplus \cdots$ yields
 $$ \mathrm{matrix}(T_h) = \left(\begin{matrix}
   I&\Pi_0T_h&\Pi_0T_h& \Pi_0T_h &\cdots \\
   0& T_h^{(1)} & \Pi_1T_h & \Pi_1 T_h &\cdots \\
   0 & 0 & T_h^{(2)} & \Pi_2 T_h & \cdots\\
   &\vdots&&\ddots
  \end{matrix}\right)$$
for any $h\geq 0$ where $I$ denotes the identity operator. We conclude that $T_h$ can be represented as an upper triangular block matrix.

We will now introduce the following three important properties of the polynomial action $T$:
\begin{definition}\label{d:loc finite}
  We say 
   \begin{itemize}
     \item $T$ is {\em locally finite} if for any $p\in\mathcal P$ there is $d\in\mathbb N$ such that for any $h>0$ the set $\{T_h^kp: k=0,\dots,d\}$ is linear dependent. Here, $T_h^k$ means composition of the operator $T_h$ $k$-times. 
     \item $T$ is {\em reducing} if for any $p\in\mathcal P$, $h\geq 0$ there is $c\in\mathbb F$ with $\deg(T_hp-cp) < \max\{1,\deg(p)\}$.
     \item $T$ is {\em strongly reducing} if there is one $c\in\mathbb F$ such that for any $p\in\mathcal P$, $h\geq 0$ one has $\deg(T_hp-cp) < \max\{1,\deg(p)\}$.
   \end{itemize}
  A polynomial process is called locally finite (resp.\ reducing, resp.\ strongly reducing) if its action is locally finite (resp.\ reducing, resp.\ strongly reducing).
\end{definition}
The following remark links the reducing property of $T$ to the classical finite-dimensional polynomial processes. Example \ref{ex: Levy} shows a process where the set of polynomials of order $1$ is infinite-dimensional and locally finiteness holds. 
\begin{remark}
  Let us consider the case that $E=\mathbb R^d$ and $\mathcal P_n$ is the set of  polynomials in $d$ commuting variables up to degree $n$. Also we assume that $X$ is an $E$-valued $\mathcal P$-polynomial process with some action $T$ and a diffusion process, i.e. $dX_t = \beta_tdt + \sigma_t dW_t$ where $W$ is a $d$-dimensional standard Brownian motion and $\beta,\sigma$ are appropriate integrands. Then \cite[Proposition 4.2.1]{cuchiero.dissertation} yields that $\beta_t = b(X_t)$ for an affine function $b:\mathbb R^d\rightarrow\mathbb R^d$ and $\sigma_t\sigma_t^T = q(X_t)$ where $q:\mathbb R^d\rightarrow \mathbb R^{d\times d}$ is a quadratic function. 
  
  Since $\mathcal P_n$ is finite dimensional, we know that $T$ is locally finite. $T$ is reducing if and only if the following two conditions hold:
     \begin{enumerate}
         \item $b(x)=b(0)+\lambda_1 x$ for some constant $\lambda_1\in\mathbb R$
         \item $q(x)=q(0)+\nabla q(0)x + \lambda_2 (x\otimes x)$ for some constant $\lambda_2\in\mathbb R$.
     \end{enumerate}
  $T$ is strongly reducing if and only if $b$ is constant and $q$ is affine.
  
  The characterisations for reducing and strongly reducing will follow easily from Corollaries \ref{c:reducing compute} and \ref{c:strongly reducing} below.
  
  We mention that also the polynomial processes discussed in \cite{BDK-polybanach} are strongly reducing.
\end{remark}

The following basic relations between reducing and locally finite actions hold:
\begin{lemma}\label{l:D=P}
  If $T$ is reducing, then it is locally finite. If $\mathcal P_n$ is finite dimensional for any $n\in\mathbb N$, then $T$ is locally finite. 
  
  If $T$ is locally finite, then the domain of the generator satisfies $\mathcal D=\mathcal P$ and $T$ can be extended to a group $(T_h)_{h\in\mathbb F}$ of linear operators such that $\deg(T_hp) \leq \deg(p)$ for any $h\in\mathbb F$, $p\in\mathcal P$ and $\deg(\mathcal Gp)\leq \deg(p)$ for any $p\in\mathcal P$.
\end{lemma}
\begin{proof}
  We first assume that $T$ is reducing. Let $p\in\mathcal P$ and denote $n:=\deg(p)+1$. We claim that $\{T_h^kp: k=0,\dots,n\}$ is linear dependent for any $h>0$. To this end, let $h>0$. Using the reducing property, we choose recursively $c_0,c_1,\dots,c_n\in\mathbb F$, $q_0,\dots,q_n\in\mathcal P$ such that
   \begin{align*}
     q_0 := p, \\
     \deg(T_hq_j-c_jq_j) &< \max\{1,\deg(q_j)\}, \\
     q_{j+1} := T_hq_j - c_jq_j
\end{align*}    
for $j=0,\dots,n-1$. Then $\deg(q_0) = n-1$, $\deg(q_{j+1}) < \max\{1,\deg(q_j)\}$ and, hence, $\deg(q_{n-1}) = 0$. Also observe from this construction that
 $$\mathrm{Span}\{T_h^kp: k=0,\dots,m\} = \mathrm{Span}\{q_0,\dots,q_{m}\}$$
 for any $m=0,\dots,n$. Since $q_{n-1}\in\mathcal P_0$ we find $T_hq_{n-1} = q_{n-1}$ and therefore we may have $c_{n-1} = 1$ and $q_n=0$. Thus we find
  \begin{align*} \mathrm{Span}\{T_h^kp: k=0,\dots,n\} &= \mathrm{Span}\{q_0,\dots,q_{n}\} \nonumber \\
  &= \mathrm{Span}\{q_0,\dots,q_{n-1}\}  \\ &= \mathrm{Span}\{T_h^kp: k=0,\dots,n-1\} \nonumber
  \end{align*} 
  which shows that $\{T_h^kp: k=0,\dots,n\}$ is linear dependent. Thus, $T$ is locally finite. 

Now, we assume that $\mathcal P_n$ is finite dimensional for any $n\in\mathbb N$ and show that $T$ is locally finite. Let $p\in\mathcal P$ and define $n:=\deg(p)$ and $d:=\dim(\mathcal P_n)$. Then the set $\{T_h^kp:k=0,\dots,d\}$ is linear dependent because it is contained in the $d$-dimensional space $\mathcal P_n$. Hence, $T$ is locally finite.
     
Finally, we only assume that $T$ is locally finite. We have $\mathcal D\subseteq\mathcal P$ by definition. Now let $p\in\mathcal P$ and choose $n\in\mathbb N$ such that $\{T_h^kp: k=0,\dots,n\}$ is linear dependent for any $h>0$. For $h>0$ we define $\mathcal P^h_p := \mathrm{Span}\{T_h^kp: k=0,\dots,n\}$. 
Then $\mathcal P^h_p$ is invariant under $T_h$ due to the linear dependence property and $\dim(\mathcal P^h_p) \leq n$ 
for any $h>0$. 
From $T_h=T^2_{h/2}$ and the $T_{h/2}$-invariance property we find that $\mathcal P^{h/2}_p\supseteq \mathcal P^h_p$. Thus,
   $$ \mathcal P^0_p := \bigcup_{l\in\mathbb N} \mathcal P^{2^{-l}}_p $$
  contains $\mathcal P^{2^{-l}}_p$ for any $l\in\mathbb N$ and due to the increase of the spaces we find
  $$ \dim(\mathcal P^0_p) = \lim_{l\rightarrow\infty}\dim(\mathcal P^{2^{-l}}_p) \leq n. $$
Therefore $\mathcal P^0_p$ is a finite dimensional subspace of $\mathcal P$ which contains $p$. Moreover, $\mathcal P^{k2^{-l}}_p\subset \mathcal P^{2^{-l}}_p$ for any $k\in \mathbb{N}$ and thus $\mathcal P^0_p$ is $T_h$-invariant for any dyadic number $h\geq0$. 
  
  We now want to see that $\mathcal P^0_p$ contains $\mathcal P^h_p$ for any $h>0$. To this end, let $h>0$ and $(\eta_l)_{l\in\mathbb N}$ be a dyadic decreasing sequence which converges to $h$. By the right-continuity property of $T$ (see Definition \ref{def:poly-action}, property (4)), we find $\mathcal P^h_p\subseteq \mathcal P^0_p$ as required. 
  Thus, $\mathcal P^0_p$ is a finite dimensional $T$-invariant subspace of $\mathcal P$ which contains $p$. 
  Hence, $\mathcal T^{(p)} := T|_{\mathcal P^0_p}$ defines a semigroup of linear operators such that $h\mapsto\mathcal T^{(p)}_hq(x)$ is right-continuous for any $x\in E$, $q\in \mathcal P^0_p$. $E$ is separating for $\mathcal P_p^0$, i.e.\ for any $q_1,q_2\in\mathcal P_p^0$ there is $x\in E$ with $q_1(x)\neq q_2(x)$ because $\mathcal P^0_p$ is a subset of functions from $E$ to $\mathbb F$. Thus, $(\delta_x)_{x\in E}$ where $\delta_x:\mathcal P^0_p\rightarrow\mathbb F$ generates the dual space of $\mathcal P^0_p$ and we see that $\mathcal T^{(p)}$ is a weakly right-continuous $c_0$-semigroup on the finite-dimensional space $\mathcal P^0_p$ (Note that due to the finite dimensionality of $\mathcal P^0_p$ all norms on it are equivalent and we simply use any norm on it). Hence $\mathcal T^{(p)}$ is norm-continuous and there is a linear operator $\mathcal G^{(p)}$ on $\mathcal P^0_p$ such that $\mathcal T^{(p)}_h = \exp(h\mathcal G^{(p)})$ for any $h>0$. We find
   $$ \frac{T_hp -p}{h} = \frac{\mathcal T^{(p)}_hp -p}{h} \underset{h\searrow0}{\longrightarrow} \mathcal G^{(p)}p$$
   and, hence, $p\in\mathcal D$ and $\mathcal Gp = \mathcal G^{(p)}p$. Also, the right-hand side of the representation
    $$ T_hp =\mathcal T^{(p)}_hp = \exp(h \mathcal G^{(p)})p $$
    yields an extension to any $h\in\mathbb F$ with the required degree invariance property.
\end{proof}

We next prove a similar result as Lemma \ref{l:T triangle} for the generator $\mathcal G$ of the action operator.
\begin{proposition}\label{p:G triangle}
  Assume that $T$ is locally finite. 
  Then we have
   $$ \mathcal G = \sum_{n=1}^\infty \sum_{l=1}^n \Pi_l \mathcal G \Pi_n = \sum_{n=1}^\infty \left(\mathcal G^{(n)}\Pi_n + \sum_{l=1}^{n-1} \Pi_l \mathcal G \Pi_n\right)$$ 
   where $\mathcal G^{(n)}:=\Pi_n \mathcal G|_{\mathcal M_n}$, and $\mathcal M_n$ and 
   $\Pi_n$ are defined in Definition \ref{def:projections}. In particular, $\deg(\mathcal Gp)\leq \deg(p)$ for any $p\in\mathcal P$.
\end{proposition}
\begin{proof}
  This is immediate from Lemmas \ref{l:D=P} and \ref{l:T triangle}, after noticing that $\Pi_l\Pi_np=0$ when $l<n$. 
\end{proof}
Rewriting the last equality in Proposition \ref{p:G triangle} as a matrix operator on $\mathcal M_0\oplus \mathcal M_1\oplus \mathcal M_2\oplus \cdots$ yields
 $$ \mathrm{matrix}(\mathcal G) = \left(\begin{matrix}
   0&\Pi_0\mathcal G&\Pi_0\mathcal G& \Pi_0\mathcal G & \cdots\\
   0& \mathcal G^{(1)} & \Pi_1\mathcal G & \Pi_1\mathcal G & \cdots \\
   0 & 0 & \mathcal G^{(2)} & \Pi_2\mathcal G &\cdots \\
   &\vdots& &\ddots
  \end{matrix}\right).$$
We remark in passing that the upper-triangular form of the generator matrix operator is analogous to what has been found for finite-dimensional polynomial processes, i.e.\ polynomial moments can be expressed in terms of matrix exponentials, see Cuchiero \cite[Theorem 4.18(ii)]{cuchiero.dissertation}.

Our next result shows that 
if $T$ is locally finite, then the semigroup $(T_h)_{h\geq 0}$ can be recovered from its generator $\mathcal G$. In principle it shows that $T_h = \exp(h\mathcal G)$ where $\exp$ is understood as a power-series and powers of $\mathcal G$ are meant in terms of composition.
\begin{proposition}\label{p:Matrix exponential}
  Assume that $T$ is locally finite. Then we have
   $$ \sum_{n=0}^\infty \frac{h^n}{n!} \mathcal G^np(x) = T_hp(x)$$
   for any $p\in\mathcal P$, $x\in E$, $h\geq 0$. Moreover, the extension as a power series converges for any $h\in\mathbb F$ and defines in this way an $\mathbb F$-parametrised group of operators.
\end{proposition}
\begin{proof}
  Lemma \ref{l:D=P} yields that $\mathcal D=\mathcal P$. Thus $\mathcal G$ is a linear operator on $\mathcal P$ and it can be applied arbitrary often to an element $p\in\mathcal P$. Let $p\in\mathcal P$ and define $\mathcal P_p := \mathrm{Span}\{\mathcal G^np: n\in\mathbb N\}$. Recall from the proof of Lemma \ref{l:D=P} that $\mathrm{Span}\{T_h^np: n\in\mathbb N\}$ is finite dimensional. Thus, $\mathcal P_p$ is finite dimensional. Choose any norm $\|\cdot\|_p$ on $\mathcal P_p$ and note that
   $$ \delta_x:\mathcal P_p\rightarrow\mathbb F, p\mapsto p(x) $$
   is a linear map defined on a finite dimensional normed space and hence continuous. On $\mathcal P_p$ we have $T_h|_{\mathcal P_p} = \exp(h\mathcal G|_{\mathcal P_p})$ where the exponential function is meant as a power series and the convergence is in operator norm. In particular we find
    $$ \sum_{n=0}^k \frac{h^n}{n!} \mathcal G^np \rightarrow T_hp,\quad k\rightarrow \infty $$
  in the normed space $\mathcal P_p$ and continuity of $\delta_x$ yields that
   $$ \sum_{n=0}^k \frac{h^n}{n!} \mathcal G^np(x) \rightarrow T_hp(x),\quad k\rightarrow \infty $$
    in $\mathbb F$. 
\end{proof}
Proposition \ref{p:Matrix exponential} is a dual-moment formula as in Cuchiero and Svaluto-Ferro \cite{CS-F}. The same is true for Corollary \ref{c:strongly reducing} below. The difference is that we require in this version locally finiteness but in return we do not need to assume the existence and uniqueness of the corresponding Cauchy problem and no additional technical conditions (in our case, existence and uniqueness follow under locally finiteness).

Since it is possible to recover the polynomial action $(T_h)_{h\geq0}$ from the generator, it is also possible to construct it in this way.
\begin{proposition}\label{p:local finite construction}
  Let $\mathcal G:\mathcal P\rightarrow \mathcal P$ be any linear operator such that
   \begin{enumerate}
       \item $\mathcal G\mathcal P_n\subseteq \mathcal P_n$ for any $n\geq 0$ and
       \item For any $p\in\mathcal P$ there is a finite dimensional subspace $V_p\subseteq \mathcal P$ such that $\mathcal G(V_p)\subseteq V_p$ and $p\in V_p$.
   \end{enumerate}
   Define $T_h$ by
     $$ T_hp(x) := \sum_{n=0}^\infty \frac{h^n}{n!} \mathcal G^np(x)$$
   for any $p\in\mathcal P$, $x\in E$, $h\geq 0$, where absolute convergence of the series is implied by the assumptions. Then $(T_h)_{h\geq 0}$ is a locally finite polynomial action.
\end{proposition}
\begin{proof}
  Let $p\in\mathcal P$ and $V_p$ be finite-dimensional $\mathcal G$-invariant containing $p$. Choose any norm $\|\cdot\|_{p}$ on $V_p$. Then $\mathcal G|_{V_p}$ is a continuous linear operator on $V_p$. Thus, its matrix exponential
   $$ \exp(h \mathcal G|_{V_p}) = \sum_{n=0}^\infty \frac{h^n}{n!} (\mathcal G|_{V_p})^n $$
  is well defined, where the convergence is in operator norm, and hence we have
   $$ T_hp(x) := \sum_{n=0}^\infty \frac{h^n}{n!} \mathcal G^np(x)$$
   is well defined with 
    $$T_h|_{V_p} = \sum_{n=0}^\infty \frac{h^n}{n!} (\mathcal G|_{V_p})^n $$ 
   for any $h\geq 0$. It follows straightforwardly from the construction that $(T_h)_{h\geq 0}$ is a locally finite polynomial action.
\end{proof}

If $T$ is reducing, then the operators $\mathcal G^{(n)}$ appearing in Proposition \ref{p:G triangle} are in fact multiples of the identity. In particular, $\mathcal G$ can be written as an upper triangular matrix.
\begin{corollary}\label{c:reducing compute}
  Assume that $T$ is reducing and recall $T_h^{(n)}$
  defined in Lemma \ref{l:T triangle} and $\mathcal G^{(n)}$ defined in Proposition \ref{p:G triangle}. 
  Then there exist constants $\lambda_n\in\mathbb F$ such that 
   \begin{align*}
      \mathcal G^{(n)} f &= \lambda_n f, \\
      T^{(n)}_hf &= e^{h\lambda_n}f, \\
      \deg(T_h f - e^{h\lambda_n}f) &< \max\{1,\deg(f)\},
   \end{align*}    
   for any $h\geq 0$, $n\in\mathbb N$, $f\in\mathcal M_n$. Note that $\lambda_0 = 0$.
\end{corollary}
\begin{proof}
  Recall from Lemma \ref{l:D=P} that $T$ is locally finite whenever it is reducing, and thus $\mathcal D=\mathcal P$. Let $f\in \mathcal M_n$. By assumption on $T$ there are constants $c_h\in\mathbb F$ for any $h\geq 0$ such that $\deg(T_hf-c_hf) < \max\{1,\deg(f)\}$. If $\deg(f) = 0$, then Proposition \ref{p:G triangle} yields that $\mathcal Gf = 0$ and Proposition \ref{p:Matrix exponential} reveals that $T_hf=f$ and so $\lambda_0=0$ satisfies the claim. 
  
  Thus we may assume that $\deg(f)\neq 0$ and, hence, $\max\{1,\deg(f)\}=\deg(f)=n$. Recall from Definition \ref{def:projections} the linear projection operator $\Pi_n:\mathcal P\rightarrow \mathcal M_n$ with kernel $\mathcal P_{n-1}\oplus \sum_{k=n+1}^\infty \mathcal{M}_k$ for $n\geq 0$, where $\mathcal P_{-1} := \{0\}$. We find that $\Pi_n(T_hf-c_hf) = 0$ because $\deg(T_hf-c_hf) < n$. Consequently, we have $T_h^{(n)}f = \Pi_nT_hf = c_hf$. By semigroup property, we find $c_{h+k}f=T_h^{(n)}T_k^{(n)} f = c_hc_k f$, which shows that $c_{h+k}=c_hc_k$. Since, for any $x\in E$, it holds,
  $$c_{h+k}f(x) = T^{(n)}_{h+k}f(x) \underset{k\rightarrow0}\longrightarrow T^{(n)}_{h}f(x) = c_hf(x)$$
 we can see that $c$ is a right-continuous semigroup of elements in $\mathbb F$. Hence, there is some $\lambda_n\in\mathbb F$ with $c_h = \exp(h\lambda_n)$. Moreover, we have for any $x\in E$
  $$ \mathcal G^{(n)}f(x) = \lim_{h\searrow 0} \frac{T^{(n)}_hf(x)-f(x)}{h} = \lambda_n f(x) $$
  and, hence, $\mathcal G^{(n)}f = \lambda_n f$.
  
  Now let $g\in\mathcal M_n$ and $\psi_n\in\mathbb F$ such that $\mathcal G^{(n)}g = \psi_ng$. We find that there is a constant $\eta_n\in\mathbb F$ with $\mathcal G^{(n)}(f+g) = \eta_n(f+g)$ and, hence,
   $$ \lambda_nf+\psi_ng = \mathcal G^{(n)}(f+g) = \eta_nf + \eta_n g.$$
  This, however, is only possible if $\lambda_n=\psi_n=\eta_n$ which shows that $\mathcal G^{(n)} = \lambda_n I$ where $I$ is the identity operator on $\mathcal M_n$. From the exponential formula in Proposition \ref{p:Matrix exponential} we find that
   $$ T_h^{(n)}f = e^{h\lambda_n}f $$
  for any $n\in\mathbb N$, $f\in\mathcal M_n$, $h\geq 0$. The result follows. 
\end{proof}

For strongly reducing polynomial processes one can give a nice formula for arbitrary polynomial moments.
\begin{corollary}\label{c:strongly reducing}
  Assume that $T$ is strongly reducing and let $\lambda_0,\lambda_1\ldots\in\mathbb F$ be as in Corollary \ref{c:reducing compute}. Then $\lambda_n = 0$ for any $n\geq 0$ and one has $\deg(T_hp-p)<\deg(p)$ for any $p\in\mathcal P\backslash \mathcal P_0$. In particular one has
   $$ T_hp = \sum_{k=0}^n \frac{h^k}{k!}\mathcal G^kp $$
  for any $n\geq 0$, $h\geq 0$ and $p\in\mathcal P_n$.
\end{corollary}
\begin{proof}
  Corollary \ref{c:reducing compute} states that $\lambda_0=0$ and $1 = e^{\lambda_0h} = c_h = e^{\lambda_nh}$ for any $h\geq 0$. Consequently, $\lambda_n=0$ for any $n\geq 0$. Thus $\deg(\mathcal Gp)<\deg(p)$ for any $p\in\mathcal P\backslash\mathcal P_0$ and $\mathcal Gp=0$ for $p\in\mathcal P_0$. Thus, we find that $\mathcal G^{n+1}p=0$ for any $p\in\mathcal P_n$. Proposition \ref{p:Matrix exponential} yields the claim.
\end{proof}

Proposition \ref{p:local finite construction} revealed how locally finite polynomial actions can be rebuilt from their generator. A similar statement can be made for strongly reducing actions and their generators.
\begin{lemma}
    Let $\mathcal G:\mathcal P\rightarrow \mathcal P$ be any linear operator such that
   \begin{enumerate}
       \item $\mathcal G\mathcal P_{n+1}\subseteq \mathcal P_n$ for any $n\geq 0$ and
       \item $\mathcal G\mathcal P_0=\{0\}$. 
   \end{enumerate}
   Define $T_h$ by
     $$ T_hp := \sum_{n=0}^\infty \frac{h^n}{n!} \mathcal G^np$$
   for any $p\in\mathcal P$, $h\geq 0$ where the summands of the series are only finitely often non-zero. Then $(T_h)_{h\geq 0}$ is a strongly reducing polynomial action.
\end{lemma}
\begin{proof}
  Let $p\in\mathcal P_n$ and observe that $\mathcal G|_{\mathcal P_n}$ is nilpotent of order $n+1$. Thus, $$\mathcal G^{k}p = 0,\quad k>n $$
  which yields the statement about the summands. Moreover,
   $$ T_hp = \sum_{k=0}^n \frac{h^n}{n!} \mathcal G^np \subseteq \mathcal P_{n-1} $$
  if $n>0$ as claimed.
\end{proof}


\section{Affine Drift for Locally Finite Polynomial Processes}\label{affine:drift:loc:finite}
Recall that $\mathcal L:=\mathcal M_1$ is a vector space complement of $\mathcal P_0$ in $\mathcal P_1$. The semigroup $T$ restricted to $\mathcal P_1=\mathcal P_0\oplus\mathcal L$ has a particularly simple structure due to Lemma \ref{l:T triangle}, 
 $$ \mathrm{matrix}(T_h|_{\mathcal P_1}) = \left(\begin{matrix}I & \Pi_0T_h \\ 0 & T_h^{(1)}\end{matrix}\right),\quad h\geq0.$$
In this section we will mostly work under the following assumption:

\vskip 0.1cm
\noindent
{\bf Assumption (F):} For every $p\in\mathcal L$, there is a finite dimensional subspace $V_p\subseteq \mathcal L$ with $p\in V_p$ which is $T^{(1)}$-invariant, i.e.\ $T^{(1)}_hV_p\subseteq V_p$ for any $h\geq 0$.
\vskip 0.1cm
\noindent

Assumption (F) means that the smallest $T^{(1)}$-invariant vector space containing $p$ is finite dimensional for any $p\in\mathcal L$.

\begin{remark}\label{r:F=lf1}
  If $T$ is locally finite, then Assumption (F) holds. If Assumption (F) holds and we forget higher order polynomials (i.e.\ we assume $\mathcal P=\mathcal P_1$), then $T$ is locally finite.
  This means that Assumption (F) is equivalent to locally finiteness of $T$ for polynomials up to order $1$. This allows, under Assumption (F), to apply most of the preceding statements for polynomials up to order $1$. 
\end{remark}

\begin{lemma}\label{l:G if F}
   Suppose Assumption (F). Then $\mathcal P_1$ is the domain of $\mathcal G|_{\mathcal P_1}$ and
    $$ T_hp(x) = \sum_{n=0}^\infty \frac{h^n}{n!} (\mathcal G|_{\mathcal P_1})^np(x) $$
    for any $p\in\mathcal P_1$, $h\geq 0$. Moreover, it has a simple matrix structure relative to the decomposition $\mathcal P_1=\mathcal P_0\oplus\mathcal L$ given by
     $$ \mathrm{matrix}(\mathcal G|_{\mathcal P_1}) = \left(\begin{matrix}0 & \mathcal G^{(0,1)} \\ 0 & \mathcal G^{(1)}\end{matrix}\right)$$
     where $\mathcal G^{(0,1)}:\mathcal L\rightarrow\mathcal P_0$ and $\mathcal G^{(1)}:\mathcal L\rightarrow\mathcal L$ are linear.
\end{lemma}
\begin{proof}
The result is a consequence of Remark \ref{r:F=lf1}, Propositions \ref{p:G triangle} and \ref{p:Matrix exponential}.
\end{proof}

Consider the algebraic dual space of $\mathcal L$, which we denote by $B$. The algebraic dual space $B$ is understood as a linearisation of $E$, where we embed the process $X$ on $B$ via
\begin{align}
\label{eq:alg-dual-embed}
(\tilde X_t)(p) := p(X_t),\quad p\in\mathcal L,t\geq 0.
\end{align}

\begin{remark}
  We believe that it is very natural to assume that $\mathcal L$ is separating for $E$, i.e.\ for any $x,y\in E$ there is $p\in\mathcal L$ with $p(x)\neq p(y)$. If this is the case, then
    $$\delta:E\rightarrow B, x\mapsto \delta_x$$ 
    is injective where $\delta_x:\mathcal L\rightarrow \mathbb F,p\mapsto p(x)$. Injectivitiy of $\delta$ means that $\tilde X = \delta(X)$ is a copy of $X$.
    
  However, our statements in this section do not require this injectivity and are stated without the separation assumption.
\end{remark}

Under assumption (F), we show in Theorem \ref{t:affine drift algebraic} below that $\tilde X$ is a weak*-martingale driven Ornstein-Uhlenbeck process. This requires a preparatory Lemma where we, in fact, construct the constant and linear part of its drift.

\begin{lemma}\label{l:integral locfin}
   Assume Assumption (F). Let $U_h:=(T_h^{(1)})^*$ for any $h\geq 0$. For $p\in\mathcal L$, $h\geq 0$ We define $U_h^p := (T_h^{(1)}|_{V_p})^*$. Then the following diagram commutes for any $p\in\mathcal L$
     $$ \begin{matrix} B & \overset{U_h}\rightarrow & B \\ \downarrow & & \downarrow \\ V_p^* & \overset{U_h^p}\rightarrow & V_p^* \end{matrix} $$
    where the mapping from $B$ to $V_p^*$ is the restriction of the linear functional on $\mathcal L$ to $V_p$.
   
Let $f:\mathbb{R}\times \mathbb{R} \rightarrow \mathbb{R}$ be continuous. Further, let $Y$ be a $B$-valued weak*-semimartingale, i.e.\ $Y(p)$ is a semimartingale for any $p\in\mathcal L$ . Then there is a $B$-valued process $Z$ such that for any $p\in\mathcal L$ one has
    $$ \int_0^t U^p_{f(t,s)}dY_s|_{V_p} = Z_t|_{V_p},\quad t\geq 0. $$
   We will also write $\int_0^t U_{f(t,s)}dY_s := Z_t$.
   
   Also, $Z$ is uniquely determined by this property in the sense that if $\hat{Z}$ is another $B$-valued process such that for any $p\in\mathcal L$ one has
   $$ \int_0^t U^p_{f(t,s)}dY_s|_{V_p} = \hat{Z}_t|_{V_p},\quad t\geq 0, $$
   then $\hat{Z}(p) = Z(p)$ $P\otimes\lambda$-a.e.,\ where $\lambda$ denotes the Lebesgue measure on $\mathbb R_+$.
\end{lemma}
\begin{proof}
   We start with the first statement. To this end, let $y\in B$, $p\in\mathcal L$ and $q\in V_p$. We find that $T_h^{(1)}q\in V_p$ and
     $$(U_hy)|_{V_p}(q) = U_hy(q) = y(T_h^{(1)}q) = y|_{V_p}(T_h^{(1)}q) = (U_h^p(y|_{V_p}))(q). $$
    Thus, $(U_hy)|_{V_p} = U_h^p(y|_{V_p})$.
    
    Now, let $Y$ be a $B$-valued weak*-semimartingale. For $p\in \mathcal L$ we have that $Y|_{V_p}$ is a semimartingale because $V_p$ is finite-dimensional. Thus, we can define
     $$ Z_t(p) := \left(\int_0^t U_{f(t,s)}^pdY_s|_{V_p}\right)(p) $$
    for any $t\geq 0$, $p\in\mathcal L$. Obviously, $Z$ has the required properties.
\end{proof}

Recall that $\mathcal G^{(0,1)}$ in Lemma \ref{l:G if F} is linear from $\mathcal L$ to $\mathcal P_0$. Hence, $(\mathcal G^{(0,1)})^*$ is linear from $\mathcal P_0^*$ to $B$, where the former is a one-dimensional space which has a natural unit, namely $\delta_e:\mathcal P_0\rightarrow\mathbb F, p\mapsto p(e)$ where $e$ is any element of $E$ (indeed, $\mathcal P_0$ is the set of constant functions from $E$ to $\mathbb F$). We will denote this unit simply by $1$.
\begin{theorem}\label{t:affine drift algebraic}
   Suppose Assumption (F) holds. Assume that $\E[\int_0^t |\mathcal G p(X_r)|dr]<\infty$ for any $t\geq 0$, $p\in\mathcal L$ and let $\mathcal G^{(0,1)}$ and $\mathcal G^{(1)}$ be as in Lemma \ref{l:G if F}. We define $b:=(\mathcal G^{(0,1)})^*1$, $\mathcal A:=(\mathcal G^{(1)})^*$, $U_h:=(T^{(1)}_h)^*$ and 
    $$ M_t := \tilde X_t - \tilde X_0 - \int_0^t (b+\mathcal A\tilde X_s) ds,\quad t\geq 0 $$
    where the integral is a weak*-integral. Then, $M$ is a $B$-valued weak*-martingale with
    \begin{align}\label{eq:Martingale Problem}
        M_t(p) = p(X_t)-p(X_0) - \int_0^t \mathcal Gp(X_s) ds, \quad t\geq 0, p\in\mathcal L 
    \end{align}    
    and we have
    $$ \tilde X_t = U_t\tilde X_0 + \int_0^t U_{t-s}bds + \int_0^t U_{t-s}dM_s, \quad t\geq 0 $$
   where the first integral is a weak*-integral, the second is the integral discussed in Lemma \ref{l:integral locfin} and 
   $$ \tilde X_t:\Omega\times \mathcal{L}\rightarrow \mathbb{F}, \tilde{X}_t(p) := p(X_t) $$
   for any $t\geq 0$.
   
   In particular, $p(X)$ is semimartingale for any $p\in\mathcal P_1$.
\end{theorem}
\begin{proof}
   Let $p\in\mathcal L$ and $0\leq s\leq t$. Since $X$ is a polynomial process we find that 
   $$\E[p(X_t)|\mathcal F_s] = T_{t-s}p(X_s).$$ 
   Recall that the domain of $\mathcal G|_{\mathcal P_1}$ is $\mathcal P_1$ due to Remark \ref{r:F=lf1}. By Lemma~\ref{l:G if F} it follows that $\partial_t T_t p(x) =\mathcal{G} T_t p(x)$ and thus by the Fundamental Theorem of Calculus 
    $$T_{t-s}p(X_s)= p(X_s) + \int_s^{t}\mathcal GT_{r-s}p(X_s)dr.$$
   By assumption we may use Fubini's theorem for the conditional expectation to find that 
    $$ \E[\int_s^{t}\mathcal Gp(X_r) dr|\mathcal F_s] = \int_s^{t}\mathcal GT_{r-s}p(X_s)dr $$
   and, hence, 
    $$ \E[p(X_t)|\mathcal F_s] = p(X_s) + \E[\int_s^{t}\mathcal Gp(X_r) dr|\mathcal F_s]. $$
   Consequently,
    $$ M^p_t := p(X_t) - \int_0^t \mathcal Gp(X_r) dr,\quad t\geq 0 $$
    is integrable at any fixed time $t\geq 0$ and its conditional expectation is given by
    $$ \E[M^p_t|\mathcal F_s] = M^p_s $$
   which yields that $M^p$ is a martingale. 
   
   By definition of $b$, we have
    \begin{align*} 
      b(p)=(\mathcal G^{(0,1)})^*1(p)= 1(\mathcal G^{(0,1)}p)
           = \mathcal G^{(0,1)}p(X_u)
   \end{align*}
   Furthermore,       
   \begin{align*}
      \mathcal A\tilde X_u(p) = \tilde X_u(\mathcal G^{(1)}p)= \mathcal G^{(1)}p(X_u)
      \end{align*}
      and therefore 
      \begin{align*}
    b(p) + \mathcal A\tilde X_u(p) = \mathcal Gp(X_u).
    \end{align*}
   Thus, we find
    \begin{align*} 
    M_t(p) &= p(X_t)-p(X_0) - \int_0^t (b+\mathcal A\tilde X_r)p ds \\
    &= p(X_t) -p(X_0) - \int_0^t \mathcal Gp(X_r) dr \\
    &= M^p_t - p(X_0)
    \end{align*}
   and consequently, $M$ is a weak*-martingale.
   
   We finalize the proof by verifying the representation of $\tilde X$. To this end, let $p\in\mathcal L$ and recall from Proposition \ref{p:Matrix exponential} that $T^{(1)}$, and henceforth $U$, can be extended to groups of operators. We find that
           \begin{align*}
          Y^p_t &:= \Big(U_t\tilde X_0 + \int_0^tU_{t-s}bds + \int_0^t U_{t-s}dM_s\Big)(p) \\&= 
            (\tilde X_0+\int_0^t U_{-s}bds + \int_0^tU_{-s}dM_s)(T_t^{(1)}p),
      \end{align*}
      where we used Lemma \ref{l:integral locfin} for the integrals with respect to $M$. The product formula for semimartingales applied on $V_p$ yields 
        $$ dY^p_t = (b+\mathcal A\tilde X_t)(p) dt + dM_t(p) = d\tilde{X}_t(p) $$
    and we have $Y^p_0 = \tilde X_0(p)$. Hence, $Y^p_t = \tilde{X}_t(p)$ as required.

\end{proof}

If $T$ is locally finite, then the right-hand side of Equation \eqref{eq:Martingale Problem} is a martingale for any polynomial $p\in\mathcal P$. 
\begin{corollary}\label{c:Martingale problem}
  Let $T$ be locally finite, $n\in\mathbb N$ and assume that $\E[\int_0^t |p(X_r)|dr]<\infty$ for any $t\geq 0$ and any $p\in\mathcal P_n$. Then the process
   $$ M_t^p := p(X_t) - \int_0^t \mathcal Gp(X_s)ds,\quad t\geq 0$$
   is an $\mathbb F$-valued martingale for any $p\in\mathcal P_n$.
\end{corollary}
\begin{proof}
  Define $\mathcal R_0:=\mathcal P_0$ and $\mathcal R_1:=\mathcal R:=\mathcal P_n$. We now consider the polynomials $\mathcal R$ where the degree of a polynomial is $1$ if it is non-constant. Note that the definition of locally finite does not involve the degrees of the polynomials and, hence, $X$ is an $\mathcal R$-polynomial process and $T$ locally finite. Moreover Assumption (F) holds relative to the polynomials $\mathcal R$. Thus, Theorem \ref{t:affine drift algebraic} states that $M^p$ is a martingale for any $p\in\mathcal R_1=\mathcal P_n$.
\end{proof}

\begin{remark}
  Corollary \ref{c:Martingale problem} establishes that all polynomial processes with locally finite action can be associated with a corresponding martingale problem. 
    Cuchiero and Svaluto-Ferro \cite{CS-F} study polynomial processes with values in Banach spaces from the martingale problem perspective.
\end{remark}

\begin{remark}
   Considering the decomposition 
      $$ \tilde X_t = \tilde X_0 + \int_0^t(b+\mathcal A\tilde X_s) ds + M_t $$
   given in Theorem 4.5 we see that $\tilde X$ is a (weak*-)martingale driven Ornstein-Uhlenbeck process with values in $B$. Recall that a $B$-valued random variable $Z$ has a weak*-expectation $\mathbb E[Z]$ if
    $$\mathbb E[|Z(p)|]<\infty,\quad p\in\mathcal L, $$
   which is the element of $B$ given by
    $$ \mathbb E[Z]:\mathcal L\rightarrow\mathbb F, p\mapsto \mathbb E[Z(p)]. $$
   If $\tilde X_t$ has weak*-expectation for every $t\geq 0$ and this expectation can be exchanged with the integral in the decomposition above, then
    $$ \mathbb E[\tilde X_t] = \mathbb E[\tilde X_0] + \int_0^t (b+\mathcal A \mathbb E[\tilde X_s]) ds, $$
i.e.\ $t\mapsto \mathbb E[\tilde X_t]$ is a solution to the ODE 
$$
u'(t)=b+\mathcal A u(t),\quad u(0)=\mathbb E[\widetilde{X}_0].
$$
This recovers the bidual formula given in \cite[Theorem 3.8]{CS-F}. 
\end{remark}

\section{Covariance Structure for Locally Finite Polynomial Processes}\label{covar:section:loc:finite}
In this section we investigate the covariance structure of polynomial diffusions. This will rely on a multiplicativity structure of the polynomials. We make the following assumption throughout this section:
\begin{assumption}
\label{ass:cov-algebra}
Suppose that, $(T_h)_{h\geq0}$ is a locally finite polynomial action.
\end{assumption}
We also recall from the previous section the set $B$ of linear functions from $\mathcal L$ to $\mathbb F$, i.e., the algebraic dual of $\mathcal L$. 
For an $E$-valued process $X$ we use (as in the previous section) its $B$-embedded version $\tilde X_t$ given by $\tilde X_t(p) := p(X_t)$.

\begin{lemma}\label{l:expectation finite}
  Let $p\in\mathcal P$ with $|p|^2\in\mathcal P$. Then we have $\E[\int_0^t |p(X_r)|dr] < \infty$ and $\E[\int_0^t |p(X_r)|^2dr] < \infty$ for any $t\geq0$.
\end{lemma}
\begin{proof}
   Let $t\geq 0$ and define $q:=|p|^2\in\mathcal P$. Note that $q(x)\geq 0$ for any $x\in E$. Since $T$ is locally finite, there is a finite dimensional vector space $V_q\subseteq \mathcal P$ with $q\in V_q$ which is $T$-invariant. Let $\|\cdot\|_q$ be a norm on $V_q$. Proposition \ref{p:Matrix exponential} yields that $r\mapsto T_rq$ is $\|\cdot\|_q$-continuous. Also, the function
    $$ \Gamma:V_q\rightarrow \mathbb F, f\mapsto \E[f(X_0)] $$
  is linear and defined on a finite dimensional space and, hence, continuous with respect to the operator norm. We find
    \begin{align*}
    \left(\E[\int_0^t |p(X_r)|dr]\right)^2 &\leq t \int_0^t \E[q(X_r)] dr\\
    &= t\int_0^t \E[T_rq(X_0)]  dr \\
    &= t \Gamma\left(\int_0^t T_rq dr \right) <\infty 
    \end{align*}
where we used Cauchy-Schwarz' inequality and Tonelli's theorem for the first inequality, the tower property with conditioning on $\mathcal F_0$ for the first equality and linearity and continuity of $\Gamma$ for the last equality and inequality.
\end{proof}

Next we state the main result of the section. It identifies the covariance coefficient of a locally finite polynomial process, which exists, and it is a second order polynomial as a function of the state. For our result we assume that the product of two first order polynomials is at most a second order polynomial and we suppose two different sufficient conditions for second order polynomials. Property (2) in the following theorem holds for classical polynomials in $d$ commuting variables and for the polynomials appearing in \cite{CS-F}.
\begin{theorem}\label{t:quadratic covariance}
 Assume that $\mathbb F=\mathbb R$, $\mathcal P_1\cdot\mathcal P_1\subseteq \mathcal P_2$ and at least one of
  \begin{enumerate}
      \item $\mathcal P_2\cdot\mathcal P_2\subseteq \mathcal P$
      \item Every element in $\mathcal P_2$ can be written as a finite linear-combination of positive elements in $\mathcal P_2$. 
  \end{enumerate}
 Let $X$ be a polynomial process with polynomial action $T$. Let $M$ be the process introduced in Theorem \ref{t:affine drift algebraic} (whose requirements are met) and $p,q\in\mathcal L$. We define
 $$a_{p,q}(x):= \mathcal G(pq)(x)-p(x)\mathcal Gq(x)-q(x)\mathcal Gp(x),\quad x\in E.$$ 
Then we have
  $$ \langle M_{(\cdot)}(p),M_{(\cdot)}(q)\rangle_t = \int_0^t a_{p,q}(X_s) ds$$
 for any $t\geq0$ where $\langle M_{(\cdot)}(p),M_{(\cdot)}(q)\rangle$ denotes the compensator of the quadratic covariation between $M_{(\cdot)}(p)$ and $M_{(\cdot)}(q)$.
\end{theorem}
\begin{proof}
   Let $r\in \mathcal P_2$. 
   
   If (1) holds, then Lemma \ref{l:expectation finite} yields that $\E[\int_0^t |r(X_s)|ds] < \infty$.
   
   If (2) holds, then there are positive $q_1,\dots,q_n\in \mathcal P_2$ and constants $c_1,\dots,c_n\in\mathbb R$ such that $r=\sum_{j=1}^nc_jq_j$. We find that
    $$ \E[\int_0^t |r(X_s)|ds] \leq \sum_{j=1}^n|c_j| \E[\int_0^t q_j(X_s) ds] = \sum_{j=1}^n|c_j| \int_0^t \E[ T_sq_j(X_0) ] ds < \infty $$
   where we used Tonelli's theorem for the equality and the last inequality follows as in the proof of Lemma \ref{l:expectation finite}.

   
   
   
   Thus, we have $\E[\int_0^t |r(X_s)|ds] < \infty$ for any $r\in\mathcal P_2$. Corollary \ref{c:Martingale problem} yields that
     \begin{align*}
         M^{p^2}_t &:= p^2(X_t) - p^2(X_0) - \int_0^t \mathcal G(p^2)(X_s)ds,\quad t\geq0, \\
         M^{p}_t &:= p(X_t) - p(X_0) - \int_0^t \mathcal Gp(X_s)ds,\quad t\geq0
     \end{align*}    
     are martingales. We have
      \begin{align*}
           (M^p_t)^2 =& p^2(X_t) + p^2(X_0) + \left(\int_0^t \mathcal Gp(X_u)du\right)^2 \\
            &\quad - 2p(X_t)p(X_0) - 2p(X_t)\int_0^t \mathcal Gp(X_u)du +2 p(X_0)\int_0^t \mathcal Gp(X_u) du \\
            =& M_t^{p^2} + \int_0^t \mathcal G(p^2)(X_s) ds + 2p^2(X_0) - 2p(X_t)p(X_0) \\
            &\quad  - 2\int_0^t p(X_u)\mathcal Gp(X_u)du  - 2\int_0^t \int_0^u\mathcal Gp(X_v)dv dM_u^p \\
            &\quad+2 p(X_0)\int_0^t \mathcal Gp(X_u) du \\
            =& M_t^{p^2} + \int_0^t a_{p,p}(X_s) ds + 2p^2(X_0)-2M_t^pp(X_0) - 2\int_0^t \int_0^u\mathcal Gp(X_v)dv dM_u^p \\
            =& N_t + \int_0^t a_{p,p}(X_s) ds
      \end{align*}
      for any $t\geq 0$ where we used the product formula for the second equality, the semimartingale decomposition from above and the fact that $[p(X),\int_0^{\cdot}\mathcal Gp(X_u)du]_t=0$ due to \cite[Proposition I.4.49(d)]{JS} where
       $$ N_t := M_t^{p^2} + 2p^2(X_0)-2M_t^pp(X_0) - 2\int_0^t \int_0^u\mathcal Gp(X_v)dv dM_u^p, \quad t\geq 0 $$
      is a martingale. Consequently,
       $$ \<M_{(\cdot)}(p),M_{(\cdot)}(p)\>_t = \int_0^t a_{p,p}(X_s) ds. $$
      Now let $q\in\mathcal P_1$. We have
       \begin{align*} 
       &\<M_{(\cdot)}(p),M_{(\cdot)}(q)\>_t \\&\quad= \frac12\left(\<M_{(\cdot)}(p+q),M_{(\cdot)}(p+q)\>_t-\<M_{(\cdot)}(p),M_{(\cdot)}(p)\>_t-\<M_{(\cdot)}(q),M_{(\cdot)}(q)\>_t\right) \\&\quad= \int_0^t a_{p,q}(X_s) ds 
       \end{align*}
     for any $t\geq 0$ as claimed.
\end{proof}

We can further sharpen our results for $E$-valued diffusions.
\begin{definition}\label{d:diffusion}
  We say that an $E$-valued process $X$ is a diffusion with drift coefficient $\beta:\Omega\times\mathbb R_+\rightarrow B$ and diffusion coefficient $\sigma:\Omega\times\mathbb R_+\rightarrow B$ if for any $p\in \mathcal L$, $\beta(p):\Omega\times\mathbb R_+\rightarrow\mathbb R$ is progressively measurable and pathwise locally integrable, $\sigma(p):\Omega\times\mathbb R_+\rightarrow\mathbb R$ is progressively measurable and pathwise locally square-integrable, and
    $$ dp(X_t) = \beta_t(p) dt + \sigma_t(p)dW_t^p $$
  for any $p\in\mathcal P$, where $W^p$ is a standard Brownian motion with values in $\mathbb R$. 
\end{definition}

Under the additional assumptions that a Brownian motion exists on the filtered probability space and that sample paths are continuous, we can show that $X$ is an $E$-valued diffusion in the sense of Definition \ref{d:diffusion}.
\begin{corollary}\label{c:lf Diffusion}
   Assume that the requirements of Theorem \ref{t:quadratic covariance} are met, $p(X)$ has continuous sample paths for any $p\in\mathcal P$ and that there is some $\mathcal{F}$-standard Brownian motion. Then there is an affine function $\beta:B\rightarrow B$ such that $X$ is a diffusion with drift coefficient $\beta(\tilde X_t)$ and some diffusion coefficient $\sigma$, and
    $$ [p(X),q(X)]_t = \int_0^t a_{p,q}(X_s)ds,\quad t\geq 0 $$
    where $[p(X),q(X)]_t$ denotes the quadratic covariation and
    $$a:\mathcal L\times\mathcal L\rightarrow\mathcal P,(p,q)\mapsto a_{p,q}:=\mathcal G(pq)-p\mathcal Gq-q\mathcal Gp.$$
\end{corollary}
\begin{proof}
 Recall from Theorem \ref{t:affine drift algebraic} $b\in B$, the linear operator $\mathcal A:B\rightarrow B$ as well as the weak*-martingale $M$. We define the affine function $\beta:B\rightarrow B,\beta(x):= b + \mathcal Ax$. By Theorem \ref{t:affine drift algebraic} we have that
    $$ dp(X_t) = \beta(\tilde X_t) dt + dM_t(p) $$
   for any $p\in\mathcal P$. Recall
   from Theorem \ref{t:quadratic covariance} 
     $$ a:\mathcal L\times\mathcal L\rightarrow \mathcal P, (p,q) \mapsto a_{p,q}:=\mathcal G(pq)-p\mathcal Gq-q\mathcal Gp. $$
    The operator $a$ is bilinear. From Theorem \ref{t:quadratic covariance} we have
     $$ [p(X),q(X)]_t = [M_{(\cdot)}(p),M_{(\cdot)}(q)]_t = \int_0^t a_{p,q}(X_s) ds,\quad t\geq0 $$
    as required. Jacod \cite[Corollaire 14.47(b)]{Jacod79} yields that $p(X)$ is a diffusion.
\end{proof}

\begin{remark}
   Say $dX_t = (\mu + \gamma X_t) dt + \sigma(X_t) dW_t$ for some Brownian motion $W$, $\mu,\gamma\in\mathbb R$ and a measurable function $\sigma$ on $\mathbb R$ of at most linear growth. One does expect that $X$ has generator
    $$ \mathcal Gf(x) = (\mu+\gamma x)f'(x) + \frac{\sigma^2(x)}{2}f''(x) $$
   For linear functions $p(x)=\alpha x$, $q(x)= \beta x$ one has
    $$ a_{p,q}(x) = (\mathcal G(pq)-p\mathcal Gq-q\mathcal Gp)(x) = \sigma^2(x) \alpha\beta$$
    which shows that $a_{p,q}$ is the quadratic covariation coefficient of $p(X)$ and $q(X)$ which appears in Corollary \ref{c:lf Diffusion}.
\end{remark}

\begin{example}
   Let $u:(x_0,x_1)\rightarrow \mathbb R$ be the unique solution of the second order ODE
    $$ u''(x) = x^2/u(x),\quad x\in (x_0,x_1) $$
   with $u(0)=1$ and $u'(0)=0$ where $x_1\in (0,\infty]$ is chosen maximally such that a solution exists, and $x_0\in [-\infty,0)$ is chosen minimally such that a solution exists. Observe that $u''$ is positive because $u$ is strictly positive on $(x_0,x_1)$. Thus $u'$ is increasing and, hence, positive on $[0,x_1)$ and negative on $(x_0,0]$. Consequently, $u$ is increasing on the positive half-line and decreasing on the negative half-line while attaining its minimal value in $0$ with $u(0)=1$. Thus we find that
    $$ u''(x) \leq x^2, \quad x\in [0,x_1). $$
   Consequently, we have $|u'(x)| \leq \frac{|x|^3}{3}$ and $u(x)\leq 1+\frac{x^4}{12}$ for any $x\in(x_0,x_1)$. This implies that $x_1=\infty$ and $x_0=-\infty$.

   Now define $\sigma(x):=\sqrt{u(x)}$ for $x\in\mathbb R$. Let $X$ be a solution to the SDE
    $$ dX_t = \sigma(X_t) dW_t,\quad X_0 = 0 $$
For polynomials $\mathcal P$ we use the span of $\{1,x,x^2,u(x)\}$, and $\mathcal L$ denotes the span of $\{x\}$. Since $\mathcal P$ is finite dimensional Proposition \ref{p:local finite construction} yields that it suffices to state the generator of the action $T$, which we define by
 $$ \mathcal Gp(x) := \frac12 u(x)p''(x),\quad p\in\mathcal P, x\in E := \mathbb R.$$
Note that $\mathcal G\mathcal P_1 =\{0\}$, $\mathcal G u(x) = \frac{x^2}{2}$ and $\mathcal G x^2 = u(x)$ for any $x\in E$ and hence the corresponding action is given by $T_hp(x) := \sum_{j=0}^\infty \frac{1}{j!}\mathcal G^jp(x)$ for any $p\in\mathcal P$, $x\in E$. Moreover, 
 $$ \E[ p(X_t) | \mathcal F_s] = T_{t-s}p(X_s) $$
 and, therefore, $X$ is a $\mathcal P$-polynomial process with action $T$. 
 However, $\mathcal P_1\cdot\mathcal P_1$ is a proper subset of $\mathcal P$.
\end{example}

\section{Affine Drift for Strongly Continuous Polynomial Processes}\label{sec: affine drift}


In this section we analyze polynomial processes under the assumptions that the polynomial action is a strongly continuous semigroup. We provide conditions which similar as in Section~\ref{affine:drift:loc:finite} under algebraic conditions allow to understand the polynomial process $X$ as a process with affine drift. Since $X$ does not take values in a linear space, again we will need to linearise $E$ first and then identify an additive decomposition of $X$ into a process which is 'mean-zero' like and an 'affine drift term'. More precisely, we aim at a decomposition of the following type
 $$ X_t = U_tX_0 + \int_0^t U_sb\ ds + R_t $$
 where the remainder term $R$ is weakly mean-zero, $(U_t)_{t\geq 0}$ is a semigroup of linear operators and $b$ is some constant. In some sense, this means that the drift of $X$ is given as $\int_0^t (b + AX_s)ds$ where $A$ may be thought of as a derivative of $U$ with respect to $t$ at time zero. Several problems will occur, though. First, $b$ can be outside the linearisation of $E$, and possibly further elements have to be added. Second, $U$ does not need to be strongly continuous but only weakly continuous and the meaning of generator is blurred. The second problem is avoided by leaving $U$ as it is and not passing to the generator in case it does not exist. 
 None of these problems occur in the classical case where $\mathcal P_1$ is finite-dimensional as in the first work on polynomial processes by Cuchiero \cite{cuchiero.dissertation}.
 Our main result of this section is summarised in Theorem \ref{t:affine drift}.

As previously done in Section~\ref{affine:drift:loc:finite}, also in this section we will not make use of the entire structure of $\mathcal P$, as only $\mathcal P_0$ and $\mathcal P_1$ will matter in our analysis. 
Doing so we might loose the algebraic property of the polynomials which, however, is not needed in this section anyway.

In this section we make the following assumption throughout: 
\begin{assumption}
\label{ass:affine}
There is a norm $\|\cdot\|$ on $\mathcal P_1$ such that:
  \begin{enumerate}
    \item $T_h|_{\mathcal P_1}$ has an extension $(\overline T_h)_{h\geq 0}$ to the completion of $(\mathcal P_1,\|\cdot\|)$ which is a strongly continuous semigroup of bounded operators.
    \item The maps $\delta_x:\mathcal L\rightarrow \mathbb F, f\mapsto f(x)$ are $\|\cdot\|$-continuous for any $x\in E$.
    \item $\mathcal P_1$ is separating for $E$, i.e.\ for any $x,y\in E$ there is $p\in\mathcal P_1$ with $p(x)\neq p(y)$.
  \end{enumerate}
\end{assumption}

We recall that $\mathcal L$ denotes a closed vector space complement of $\mathcal P_0$ in $\mathcal P_1$ (also denoted $\mathcal M_1$, the first-order monomials in our setting). We also note in passing that if $(E,\Vert\cdot\Vert_E)$ is a Banach space and $\mathcal P_1$ a subset of linear functionals, then the norm on $\mathcal P_1$ is naturally chosen to be the operator norm on linear functionals, i.e., the elements in the dual of $E$. 


 The dual space of $\mathcal L$ is denoted by $(B,\|\cdot\|_B)$, i.e.\ $B$ is the set of $\|\cdot\|$-continuous linear maps on $(\mathcal L,\|\cdot\|)$. The space $B$ will play the role as the linearisation of $E$. One sees that the dual space $\mathcal P_1^*$ of $\mathcal P_1$ is isomorphic to $\mathbb F\times B$ with norm $\|(c,b)\|_{\mathbb F\times B} := |c| + \|b\|_B$. The extra $\mathbb F$-dimension, which is generated by the constant functions, does not play a substantial role in our analysis to come.

The set $E$ has a natural embedding into $B$ because $\delta_x$ as a functional on $\mathcal L$ is linear and continuous by property (2) of Assumption \ref{ass:affine}, that is, $\delta_x\in B$. Moreover, $\mathcal P_1$ is separating for $E$ and, hence, so is $\mathcal L$. I.e.,  $\delta_x\neq\delta_y$ for any $x,y\in E$ with $x\neq y$, which implies that
$$ \delta:E\rightarrow B, x\mapsto \delta_x $$
is an injective map. We denote the embedding of $X$ into $B$ by 
\begin{align}\label{a:embedding to B}
\tilde X_t:=\delta_{X_t},\qquad t\geq 0.
\end{align}
This is how $B$ will be interpreted as a linearisation of $E$.
We emphasise in passing that in the previous section, $B$ denoted the {\it algebraic} dual, while now it is the space of {\it continuous} linear functionals on $\mathcal L$. 
\begin{remark} 
  If $(E,\|\cdot\|_E)$ is a reflexive Banach space and $(\mathcal L,\|\cdot\|)$ its dual, then $E$ is isometrically isomorphic to $B$ where the embedding is $\delta:E\rightarrow B,x\mapsto \delta_x$. 
  Also, note that we assumed $(\overline T_h)_{h\geq 0}$ to be strongly continuous
  (by Property (1) in Assumption \ref{ass:affine}). However, its dual semigroup on $\mathbb F\times B$ does not need to be strongly continuous. Consider the following example: The left-shift semigroup $U$ on $L^1(\mathbb R)$ is generated by the weak derivative, i.e.\ $\mathcal Gf = f'$, where $f\in L^1(\mathbb R)$ is absolutely continuous with absolutely continuous derivative $f'\in L^1(\mathbb R)$. The dual of $L^1(\mathbb R)$ is isometric to $L^\infty(\mathbb R)$ with $\<g,f\>=\int_{\mathbb R} f(x)g(x) dx$ for $f\in L^1(\mathbb R)$ and $g\in L^\infty(\mathbb R)$. The dual of the left-shift $U$ is the right-shift $R$ on $L^\infty(\mathbb R)$, which is not strongly continuous because $R_t 1_{[0,1]} \not\rightarrow 1_{[0,1]}$ in norm (the norm being the uniform norm) as $t\searrow 0$. $\mathcal G^* = -\partial_x$ on its domain. (This is an example with a strongly continuous group!)
\end{remark}

\begin{remark}
  For any $p\in\mathcal P_1$, property (3) in Definition \ref{def:poly-process} yields that the $\mathbb F$-valued process $p(\tilde X)$ has c\`adl\`ag paths. Notice that for $p\in\mathcal L$, we mean by $p(\tilde{X}_t)=\tilde{X}_t(p)=p(X_t)$. Thus, the notation $p(\tilde{X})$ makes use of the identification of $E$ with $B$. For general $p\in\mathcal P_1$, we use that the dual space of $\mathcal P_1$ is isomorphic to $\mathbb F\times B$, as noted above.  
\end{remark}

We start by inspecting the structure of the semigroup $\overline T$. We use the notations 
\begin{align}\label{eq:T-bar-projection}
 \overline T_h^{(i)} := \Pi_i\overline T_h|_{\overline{\mathcal L}}, \qquad i=0,1,   
\end{align}
where $\Pi_0,\Pi_1$ are projectors on $\overline{\mathcal P}_1$ with $\Pi_0+\Pi_1$ equal to the identity and with ranges $\mathcal P_0$ and $\overline{\mathcal L}$, respectively. Lemma \ref{l:T triangle} implies on $\overline{\mathcal P}_1 = \mathcal P_0\oplus \overline{\mathcal L}$ that 
\begin{equation}
\label{eq:bar-T-matrix}
\mathrm{matrix}(\overline T_h) = \left(\begin{matrix} I & \overline T_h^{(0)}\\ 0 & \overline T^{(1)}_h \end{matrix}\right).
\end{equation}
This implies that the generator $\overline{\mathcal G}$ of $\overline T$ has the structure
 $$ \mathrm{matrix}(\overline{\mathcal G}) = \left(\begin{matrix} 0 & \overline{\mathcal G}^{(0)}\\ 0 & \overline{\mathcal G}^{(1)} \end{matrix}\right) $$
where $\overline{\mathcal G}^{(1)}$ is the generator of $\overline T^{(1)}$ and $$
\overline{\mathcal G}^{(0)} :=\{(f,c)\in\overline{\mathcal L}\times\mathcal P_0: c = \lim_{h\searrow0}\frac1h\overline T_h^{(0)}f\}.$$ 
Equation \eqref{eq:T-bar-projection} yields that the dual operator $U_h$ of $\overline T_h$ has the presentation
 $$ \mathrm{matrix}(U_h) = \left(\begin{matrix} I & 0 \\ U_h^{(0)} & U^{(1)}_h \end{matrix}\right) $$
 on $\mathbb F\times B$ for any $h\geq 0$ where $U_h^{(1)} = (\overline T_h^{(1)})^*$ and $U_h^{(0)} = (\overline T_h^{(0)})^*$.

\begin{lemma}\label{l:Generator V}
  Let $\overline{\mathcal D}$ be the domain of $\overline{\mathcal G}$ and $\overline{\mathcal D}^{(1)}$ be the domain of $\overline{\mathcal G}^{(1)}$. Then we have 
   $$ \overline{\mathcal D} = \mathcal P_0\oplus \left(\overline{\mathcal D}\cap\overline{\mathcal L}\right), \quad\quad \overline{\mathcal D}\cap\overline{\mathcal L} \subseteq\overline{\mathcal D}^{(1)}$$
  and for any $f = c+\ell\in \mathcal P_0\oplus \left(\overline{\mathcal D}\cap\overline{\mathcal L}\right)$ we have
   $$ \overline{\mathcal G}f = \overline{\mathcal G}^{(0)} \ell + \overline{\mathcal G}^{(1)}\ell. $$
\end{lemma}
\begin{proof}
  Clearly, $\mathcal P_0\subset\overline{\mathcal D}$ because for any $c\in\mathcal P_0$,
   $$ \frac{\overline{T}_hc-c}{h} = \frac{T_hc-c}{h} = 0, 
   $$
  and we find that $\overline{\mathcal G}c=0$. Thus, we have $\overline{\mathcal D} \supseteq \mathcal P_0\oplus \left(\overline{\mathcal D}\cap\overline{\mathcal L}\right)$. Let $f\in\overline{\mathcal D}$, $\ell:=\Pi_1f$ and $c:=f-\ell\in\mathcal P_0$. Since the domain of the generator is a vector space we find that $\ell\in\overline{\mathcal D}$. Consequently, we have 
  $$ \overline{\mathcal D} = \mathcal P_0\oplus \left(\overline{\mathcal D}\cap\overline{\mathcal L}\right).$$

  Now, let $\ell\in\overline{\mathcal D}\cap\overline{\mathcal L}$. Then we have
   $$ \frac{\overline{T}^{(1)}_h\ell-\ell}{h} = \Pi_1\frac{\overline{T}_h\ell-\ell}{h} \rightarrow \Pi_1\overline{\mathcal G}\ell $$
   and hence $\ell\in\overline{\mathcal D}^1$ and $\overline{\mathcal G}^{(1)}\ell = \Pi_1\overline{\mathcal G}\ell$.
\end{proof}

\begin{remark}
   Let $M$ be an $\mathbb R^d$-valued martingale, $A\in\mathbb R^{d\times d}$ and $b\in\mathbb R^d$. Then the process given by
    $$ Y_t = Y_0 + \int_0^t b+AY_s ds + M_t, \quad t\geq 0$$ 
   satisfies
     $$ Y_t = V_tY_0 + \int_0^tV_s b ds + \int_0^t V_{t-s}dM_s, \quad t\geq 0 $$
where $V_t = \exp(tA)$ for any $t\geq0$. These two representations are of course equivalent in finite dimension. On a more general structure, the second representation of $Y$ might not be good enough to recover the first. $V$ might be non-differentiable, and stochastic integration might be ill-defined so that only the process $R_t = \int_0^t V_{t-s}dM_s$ is available. This is the situation we describe in the next theorem.
\end{remark}

We come to our main result in this Section, showing the affine drift of polynomial processes. 
\begin{theorem}\label{t:affine drift} 
  There is a Banach space $(B^+,\|\cdot\|_{+})$ which contains $B$ as a sub-vector space, and an element $b\in B^+$. Moreover, the dual semigroup of $\overline{T}^{(1)}$ can be extended uniquely to a semigroup $U$ of bounded operators on $B^+$ such that
   $$ \tilde X_t = U_t\tilde X_0 + \int_0^t U_sbds + R_t $$
   where $R:\Omega\times\mathbb R_+\rightarrow B^+$ with $\E[\ell(R_t)] = 0$ for any $\ell\in\overline{\mathcal L}$. The integral in the representation above is understood in a weak*-sense, i.e.\ $\int_0^t U_sb\ ds$ is the unique element $c_t\in B$ (which exists) with $p(c_t) = \int_0^t p(U_sb) ds$ for any $p\in \mathcal C$ where $\mathcal C$ is the domain of $\overline{\mathcal G}$ restricted to $\overline{\mathcal L}$, i.e., $\overline{\mathcal D}\cap\overline{\mathcal L}$ being a dense subspace of $\overline{\mathcal L}$.
   
  If additionally, $\overline{T}$ can be extended to a group of bounded operators, then $M_t := (\overline{T}^{(1)}_{-t})^*R_t$ is a weak*-martingale, i.e.\ for any $\ell\in\overline{\mathcal L}$ one has that
   $\ell(M)$
  is an $\mathbb F$-valued martingale. In particular, one has
   $$ \tilde X_t = U_t\tilde X_0 + \int_0^t U_sbds + U_tM_t,\quad t\geq 0. $$
   
  If the domain of $\overline{\mathcal G}$ is $\overline{\mathcal{L}}$, then $\overline T$ can be extended to a group of bounded operators and $(B^+,\|\cdot\|_+)$ can be chosen to be equal to $(B,\|\cdot\|_B)$.
\end{theorem}
\begin{proof}
  {\em Construction of the extension space $B^+$:}
    Let $\mathcal C$ be the domain of the generator $\overline{\mathcal G}$ restricted to $\overline{\mathcal L}$ and define the norm 
    $$\|f\|_\mathcal {C}:= \|f\| + \|\overline{\mathcal G}f\|,\quad f\in\mathcal C.$$
    Then $(\mathcal C,\|\cdot\|_\mathcal {C})$ is a Banach space and we denote its dual space by $(B^+,\|\cdot\|_{+})$. Moreover, Lemma \ref{l:Generator V} yields that $\mathcal C$ is $\|\cdot\|$-dense in $\overline{\mathcal L}$.
    For orientation we have the following diagram:
   $$ \begin{matrix} (\mathcal L,\|\cdot\|) & \text{has dual space} & (B,\|\cdot\|_{B}) \\ \rotatebox[origin=c]{90}{$\subseteq$} && \rotatebox[origin=c]{270}{$\subseteq$} \\ (\mathcal C,\|\cdot\|_{\mathcal C}) & \text{has dual space} & (B^+,\|\cdot\|_{B^+}) \end{matrix}$$
   We have $\mathcal C$ is invariant under $\overline T^{(1)}$ and $\overline T^{(1)}$ restricted to $\mathcal C$ is bounded with respect to the $\|\cdot\|_{\mathcal C}$-operator norm. We denote the dual semigroup of operators of $(\overline T_h^{(1)}|_{\mathcal C})_{h\geq 0}$ by $U$, i.e.\ $(U_h)_{h\geq 0}$ is a semigroup of bounded operators on $B^+$. Note that $U_h|_B$ is the dual of $\overline T^{(1)}_h$ for any $h\geq 0$.
 
 {\em Construction of the constant drift part $b\in B^+$:} 
   We have $\overline{\mathcal G}^{(0)}$ is an operator from $\mathrm{dom}(\overline{\mathcal G})$ to $\mathcal P_0$, and therefore find that $\overline{\mathcal G}^{(0)}|_{\mathcal C}:\mathcal C\rightarrow \mathcal P_0$. Hence, for its dual operator we have $(\overline{\mathcal G}^{(0)}|_{\mathcal C})^*:(\mathcal P_0)^*\rightarrow B^+$, where $(\mathcal P_0)^*$ is the dual of the one-dimensional space $\mathcal P_0$. Since $\mathcal P_0$ is one-dimensional, there is $1\in (\mathcal P_0)^*$ such that $1(p) = p(x)$ for any $x\in E$. We define $b:= (\overline{\mathcal G}^{(0)}|_{\mathcal C})^*1$.

 {\em Existence of the weak*-integral on $B^+$ with values in $B$:} Let $p\in\mathcal C$ and $s\geq 0$. We have $\Pi_0\overline{\mathcal G}\overline{T}_s p \in \mathcal P_0$ and, hence, a constant function. Thus for any $x\in E$ we find that $f_p(s) := (\Pi_0\overline{\mathcal G}\overline{T}_s p)(x)$ is its value. By the $c_0$-semigroup property of $T$ we find that $s\mapsto f_p(s)$ is continuous and, hence Lebesgue-integrable. Also, observe that $p(U_s b) = f_p(s)$ by duality. 

 We have 
  \begin{align}
    \int_0^t p(U_s b) ds &= \int_0^t f_p(s) ds \nonumber \\
                         &= \Pi_0 \int_0^t\overline{\mathcal G} \overline{T}_sp ds (x) \nonumber\\
                         &= \Pi_0 (T_tp-T_0p) (x) \nonumber\\
                         &= (T_t^{(0)}-\Pi_0)p(x) \nonumber\\
                         &= T_t^{(0)}p(x),\label{eq:drift-duality-repr}
\end{align}   
 where the integral in the second line is a Bochner-integral in $(\overline{\mathcal L},\|\cdot\|)$ and the point $x\in E$ is arbitrary because both $T^{(0)}_t$ and $\Pi_0$ map into the constant functions $\mathcal P_0$. The last equality holds because simply $\Pi_0(\overline{\mathcal L})=\{0\}$. The functional $c_t:\overline{\mathcal L}\rightarrow \mathbb F,p\mapsto (T_t^{(0)}-\Pi_0)p(x)$ is continuous linear and hence an element of $B\subseteq B^+$. From the equation above we find that
  $$ \int_0^t p(U_sb) ds = p(c_t) $$
  and therefore $c_t=\int_0^t U_sb ds$ as a weak*-integral on $B^+$.
 
 {\em Construction of $R$:}
  We define $R_t := \tilde X_t - U_t\tilde X_0 - \int_0^t U_sb ds$ for any $t\geq 0$. Since $B$ is $U$-invariant and by the above argument $\int_0^tU_sb ds \in B$, we find that $R_t\in B$ for any $t\geq 0$, $P$-a.s. We next show that $\E[|\ell(R_t)|]<\infty$ for $\ell\in\overline{\mathcal L}$. We have
  $$
  \ell (R_t) = \ell(\tilde X_t)-\ell(U_t\tilde X_0)-\ell(\int_0^tU_sbds),
  $$
  and the two last summands on the right hand side are obviously absolutely integrable. The first summand is $\ell(\tilde X_t) = \ell(X_t)$, which is absolutely integrable by assumption on $X$. (Recall Property (1) in Definition \ref{def:poly-process}). 
  
  Let $p\in\mathcal C$. We have by polynomial property of $X$
  \begin{align*}
    \E[ p(R_t) ] &= \E\left[ p(\tilde X_t) - T^{(1)}_tp(\tilde X_0) - p\left(\int_0^tU_sbds\right) \right] \\&= T_tp(\tilde X_0) - T^{(1)}_tp(\tilde X_0) - p\left(\int_0^tU_sbds\right) \\
    &= T_tp(\tilde X_0) - T^{(1)}_tp(\tilde X_0) - T_t^{(0)}p(\tilde{X}_0) \\
    &= 0.
\end{align*}
 Since $R$ is $B$-valued and we have the required identity on $\mathcal C$, we can extend it to $\overline{\mathcal L}$ by a density argument. Also note that $R_0 = 0$.

 {\em Martingale representation when $\overline T$ can be extended:} We now assume that $\overline T$ can be extended to a group of bounded operators and let $p\in\mathcal C$. Note that $(\overline T_{-t}^{(1)})^*$ is the inverse of $U_t$ and we will simply denoted it as $U_{-t}$. Obviously, $(U_t)_{t\in\mathbb R}$ is a group of operators on $B^+$. We find that
    $$ M_t = U_{-t}R_t = U_{-t}\tilde X_t - \tilde X_0 - U_{-t}\int_0^tU_sbds. $$
   Moreover, 
    $$ U_{-t}\int_0^tU_sbds = \int_0^t U_{s-t}b ds $$
   because for $q\in \mathcal C$ one has $\overline T_{-t}^{(1)}q \in\mathcal C$ and
    $$ q(U_{-t}\int_0^tU_sbds) = \overline T_{-t}^{(1)}q(\int_0^tU_sbds) = \int_0^t \overline T_{-t}^{(1)}q(U_sb) ds = \int_0^t q(U_{s-t}b) ds. $$
   Also we have by the polynomial property of $X$
   \begin{align*}
     \E[p(U_{-t}\tilde X_t)|\mathcal F_s] &= \E[\overline T_{-t}^{(1)}p(\tilde X_t)|\mathcal F_s] \\
        &= \overline T_{t-s}\overline T_{-t}^{(1)}p(\tilde X_s) \\
        &= \overline T^{(0)}_{t-s}\overline T_{-t}^{(1)}p(\tilde X_s) + \overline T_{-s}^{(1)}p(\tilde X_s) \\
        &= \int_0^{t-s} \overline T_{-t}^{(1)}p(U_rb) dr + p(U_{-s}\tilde X_s) \\
        &= \int_0^{t-s} p(U_{r-t}b) dr + p(U_{-s}\tilde X_s).
   \end{align*}
   In the third equality we used the matrix representation of $\overline T$ in \eqref{eq:bar-T-matrix}, and in the fourth equality we make use of the representation \eqref{eq:drift-duality-repr}. 
   Hence, we find that 
   \begin{align*}
     \E[p(M_t)|\mathcal F_s] &= \int_0^{t-s} p(U_{r-t}b) dr + p(U_{-s}\tilde X_s) - p(\tilde X_0) - \int_0^tp(U_{r-t}b)dr \\
      &= p(M_s)
   \end{align*}
   as required. Since $M$ is $B$-valued we find that $\E[p(M_t)|\mathcal F_s] = p(M_s)$ for any $p\in \overline{\mathcal L}$. The Theorem is proven. 
\end{proof}

\begin{corollary}\label{c:affine drift}
   Under the assumptions and notations of Theorem \ref{t:affine drift} and for fixed $s\geq0 $ there is a progressively measurable process $(R^s_t)_{t\geq0}$ with $\E[R^s_t|\mathcal F_s] = 0$ for any $t\geq 0$ and
    $$ \tilde X_t = U_{t-s}\tilde X_s + \int_s^t U_rbdr + R^s_t $$
   for any $t\geq s$.
\end{corollary}
\begin{proof}
   Define $R^s_t = 0$ for $t\in[0,s]$ and $Y_u := X_{u+s}$ for $u\geq 0$. $Y$ is a polynomial process with action $T$ relative to the filtration $(\mathcal F_{u+s})_{u\geq 0}$. According to Theorem \ref{t:affine drift} there is a mean-zero process $(R_u)_{u\geq 0}$ which is progressively measurable with respect to the filtration $(\mathcal F_{u+s})_{u\geq 0}$ such that
    $$ Y_u = Y_0 + \int_0^u U_rbdr + R_u $$
for any $u\geq 0$. Define $R^s_t := R_{t-s}$ for $t>s$. The claim follows.
\end{proof}

\begin{remark}
  Note that the decomposition
   \begin{equation}
   \label{eq:mild-sol}
       \tilde X_t = U_t\tilde X_0 + \int_0^t U_rb dr + U_tM_t
       \end{equation}
  which appears in Theorem \ref{t:affine drift} when $T$ can be extended to a $c_0$-group implies that
   $$ \tilde X_t = U_{t-s}\tilde X_s + \int_0^{t-s}U_rbds + U_{t}(M_t-M_s)$$
   for any $0\leq s\leq t$. 
   
If the dual $A$ of $\overline{\mathcal G}^{(1)}$ is densely defined and generates a $c_0$-semigroup (which then is $U$), the expression in  \eqref{eq:mild-sol} is the mild solution (see Peszat and Zabczyk \cite{peszat.zabczyk.07}) of  
    $$ d\tilde X_t = (b+A\tilde X_t) dt + dN_t $$
    where $N_t := \int_0^t U_{-s} dM_s$. This holds true whenever we have available a martingale integration theory (see for example van Neerven \cite{vNeerven} for stochastic integration on Banach spaces).
    Thus, $\tilde X$ is an $N$-driven Ornstein-Uhlenbeck process, a formula which is true for locally finite polynomial processes, cf.\ Theorem \ref{t:affine drift algebraic}. 
\end{remark}

The following example shows that the constant part of 'the drift' of a polynomial process can actually point outside the space $B$, i.e.\ $b\in B^+\backslash B$. 
This is only possible if the semigroup $U$ together with the Bochner integral "smoothen" the constant drift vector back into the space $B$ because we know that $\int_0^tU_sbds \in B$ from Theorem \ref{t:affine drift}.
\begin{example}\label{ex:outside drift}
  Let $E=l^2(\mathbb N,\mathbb C)$ equipped with the norm $\|x\|^2_E:=\sum_{n\in\mathbb N}|x|^2$, $(\mathcal L,\|\cdot\|)$ be the dual space and $\mathcal P$ be the algebra generated by $\mathcal L$ where the multiplication is the image-wise product. Since $(E,\|\cdot\|_E)$ is reflexive, it is (up to isometric isomorphism) the dual space of $(\mathcal L,\|\cdot\|)$. In particular, we have $B \cong E$.
 Let $W$ be the Wiener process with covariance operator $Q$ on $E$ given by $(Qx)_n := \frac{x_n}{(1+n)^2}$ for $x\in E$. Let $(Ax)_n := 2\pi in x_n$. $A$ is a normal operator and it generates the $c_0$-group $$(U_hx)_n = e^{2\pi inh}x_n.$$ Moreover, $A^* = -A$, which generates the $c_0$-group $(U_{-h})_{h\in\mathbb R}$. Note that if $x\in \dom(A^*)=\dom(A)$, then $\sum_{n\in\mathbb N} |x_n| < \infty$ because $Ax\in B$ and, hence, $x_n = (Ax)_n \frac{1}{2\pi i n}$ is the product of two elements in $B$, from which it follows $\dom(A)\subseteq l^1(\mathbb N,\mathbb R)$. Now define $\Gamma:\dom(A^*)\rightarrow\mathbb C,x\mapsto \sum_{n\in\mathbb N}x_n$ which corresponds to the element $$b := (1)_{n\in\mathbb N} \in l^\infty(\mathbb N,\mathbb C) =: B^+.$$ $U$ extends naturally to $B^+$ and
   $$ \left(\int_0^t U_sb\ ds\right)_n = \frac{e^{2\pi int}-1}{2\pi in},\quad t\geq 0,n\in\mathbb N,n\geq 1. $$
   Thus, $\int_0^t U_sbds \in E$. We now define
    $$ X_t := \int_0^t U_sb ds + \int_0^t U_{t-s}dW_s,\quad t\geq 0 $$
   which is an $E$-valued polynomial process which can be interpreted as the mild solution to the SPDE
    $$ dX_t = (b+ AX_t) dt + dW_t. $$
    We refer to Peszat and Zabczyk \cite{peszat.zabczyk.07} for mild solutions of SPDEs.
\end{example}
It is interesting to notice that the state space $E$ of the polynomial process introduced in the Example above is a Hilbert space. Even for such nice state spaces we may have polynomial processes where the drift $b$ is outside the state space. On the other hand, this can only happen when the semigroup $U$ is sufficiently regular.

We close this section with the following corollary.
\begin{corollary} 
  Suppose 
  $(\mathcal P_1,\|\cdot\|)$ is a Banach space and $T$ is locally finite.
  Then $T$ extends to a group and we denote $U$ to be the dual group of $T|_{\mathcal L}$. 
  
  Furthermore, there is $b\in B$ and a weak*-martingale $M$ such that
    $$ \tilde X_t = U_t\tilde X_0 + \int_0^t U_sb ds + U_t M_t. $$
    
  Also, $U$ is differentiable with respect to $t$ and its generator $A:=\partial_tU_t|_{t=0} \in L(B)$ is the dual of $\mathcal G$. 
\end{corollary}
\begin{proof}
   Proposition \ref{p:Matrix exponential} yields that $T$ can be extended to a group of operators on $\mathcal P_1$. Also, $T$ is strongly continuous because by assumption for any $p\in \mathcal P_1$ there is a finite dimensional space $V_p\subseteq \mathcal P_1$ with $p\in V_p$ and $T_t(V_p)\subseteq V_p$ for any $t\geq 0$. Proposition \ref{p:Matrix exponential} yields strong continuity of $T_t|_{V_p}$ and that its generator has domain $V_p$. Consequently, $T$ is strongly continuous and its generator has domain $\mathcal P_1$. Then its generator is a bounded operator and
    $$ T_t = \exp(t \mathcal G) := \sum_{n=0}^\infty \frac{\mathcal G^n}{n!}.$$
  
  Its dual group $U$ satisfies
   $$ U_t = \exp(t A) $$
   where $A$ is the dual of $\mathcal G$ and the claim follows from Theorem \ref{t:affine drift}.
\end{proof}

\section{Covariance Structure of Polynomial Diffusions}
Similar as in Section~\ref{covar:section:loc:finite} for locally finite polynomial processes we shall now analyse the covariance structure for polynomial diffusions under continuity assumptions. 
  \begin{assumption}
\label{ass:affine extended}
There is a complete norm $\|\cdot\|$ on $\mathcal P_4$ such that:
  \begin{enumerate}
    \item $(T_h|_{\mathcal P_4})_{h\geq 0}$ is a strongly continuous semigroup of bounded operators.
    \item The maps $\delta_x:\mathcal P_4\rightarrow \mathbb F, f\mapsto f(x)$ are $\|\cdot\|$-continuous for any $x\in E$.
    \item The linear functional $\Gamma:\mathcal P_4\rightarrow \mathbb F,p\mapsto \E[p(X_0)]$ is continuous.
    \end{enumerate}
  Also we assume that: 
   \begin{enumerate}
   \setcounter{enumi}{3}
    \item $\mathcal P_1$ is separating for $E$, i.e.\ for any $x,y\in E$ there is $p\in\mathcal P_1$ with $p(x)\neq p(y)$.
     \item $\mathbb F=\mathbb R$
       \item $\mathcal P_1\cdot\mathcal P_1 \subseteq\mathcal P_2$ and $\mathcal P_2\cdot\mathcal P_2 \subseteq\mathcal P_4$.
   \end{enumerate}
\end{assumption}
Recall that the generator $\mathcal G$ had been defined (in Definition \ref{def:poly-action}) in a way which does not use the topology on $\mathcal P_4$. However, if $p$ is in the domain $\mathcal C$ of the generator of the $c_0$-semigroup $(T_h|_{\mathcal P_4})_{h\geq 0}$, i.e.\ there is $q\in\mathcal P_4$ such that $$ q = \lim_{h\searrow 0} \frac{T_hp-p}{h}, $$
then $p\in\mathcal D$ (which was defined in Definition \ref{def:poly-action}) because $\delta_x$ is continuous which yields
 $$ q(x) = \lim_{h\searrow 0} \frac{T_hp(x)-p(x)}{h},\quad x\in E. $$
 In this case we also have $\mathcal Gp = q$ which reveals that $\mathcal G|_{\mathcal C}$ is the generator of the $c_0$-semigroup $(T_h|_{\mathcal P_4})_{h\geq 0}$. Also, note that $\mathcal Gp\in\mathcal P_n$ for any $p\in\mathcal C\cap \mathcal P_n$ for $n=0,\dots,4$ because $\mathcal P_n$ is a closed $T$-invariant space.

\begin{remark}[Incomplete norm]
  In Assumption \ref{ass:affine extended} we assumed that the norm is complete. If the norm is not complete one would like to pass to the completion and replace $\mathcal P_n$ with its completion for $n=1,\dots,4$. 
  
  This is unproblematic for assumption (2) to (5) as $\delta_x$ and $\Gamma$ can be extended to bounded linear maps on the completion. The continuous extension of $T_h$ to the completion is still a bounded operator and the family of continuations is still a semigroup. However, the family of continuations is strongly continuous if and only if $\sup_{h\in[0,1]}\|T_h\|_{\mathrm{op}}$ is bounded. Finally, assumption (6) does not necessarily carry over to the completion. A sufficient condition to still hold on the completion is that the multiplication from $\mathcal P_2\times \mathcal P_2\rightarrow \mathcal P_4$ is a bounded bilinear map. 
  
  Finally we like to note that the degree of some polynomials might be lower after extending to the completion (since it is possible that some elements of $\mathcal P_n$ can be approximated by elements in $\mathcal P_{n-1}$ for $n=2,3,4$) but this does not pose any problem.
\end{remark}


We first state a simple consequence of the fundamental theorem of calculus for the semigroup $T$.
\begin{lemma}\label{l:semigroup diff}
   Let $p\in\mathcal C$ 
   Then we have
    $$ T_tp = T_sp + \int_s^t \mathcal G T_rp dr,\quad t\geq s\geq 0. $$
\end{lemma}
\begin{proof}
   We have $T_tp\in\CC$ and 
    $$ \partial_t T_tp = \GG T_tp $$
  for any $t\geq 0$. Since $\GG|_{\mathcal C}$ is closed we have that 
   $$ t\mapsto \GG T_tp $$
  is continuous. The Fundamental Theorem of Calculus yields the claim.
\end{proof}
We also have a useful martingale-result for a class of polynomials $p\in\mathcal P_2$:
\begin{lemma}\label{l:Mp martingale}
  Let $p\in\CC\cap \mathcal P_2$. Then 
    $$ M^p_t := p(X_t) - \int_0^t \GG p(X_r) dr,\quad t\geq 0 $$
  defines a martingale with $\E[|M^p_t|^2]<\infty$ for any $t\geq 0$.
\end{lemma}
\begin{proof}
   Define $q:=(\GG p)^2$ and observe that $q\in\mathcal P_4$ by Assumption \ref{ass:affine extended}(6) with $q(x)\geq 0$ for any $x\in E$. Let $0\leq s\leq t$. We get from
   Cauchy-Schwarz' inequality 
   that
    \begin{align*} 
    \E[|\int_0^t \mathcal Gp(X_r)dr|^2] &\leq \E[t\int_0^t|\mathcal Gp(X_r)|^2dr ] \\ &\leq  t\E[\int_0^t q(X_r) dr]\\ &= t \int_0^t\mathbb E[T_rq(X_0)] dr \\
     &= t\Gamma\left(\int_0^t T_rq dr\right) <\infty.
    \end{align*}
    In the first equality above we applied the tower property with conditioning on $\mathcal F_0$ and in the last bound Assumption \ref{ass:affine extended}(3).
   Hence, $M^p_t$ has finite expectation. Fubini's theorem for the conditional expectation yields
    $$ \E[\int_0^t \mathcal Gp(X_r) dr|\mathcal F_s] = \int_0^s \mathcal Gp(X_r)dr + \int_s^t \mathcal GT_{r-s}p(X_s)dr. $$
   Lemma \ref{l:semigroup diff} yields
     $$ \int_s^t \mathcal GT_{r-s}p(X_s)dr = T_{t-s}p(X_s)-p(X_s). $$
    Thus,
     $$ \E[M^p_t|\mathcal F_s] = T_{t-s}p(X_s) - \left( \int_0^s \mathcal Gp(X_r)dr+T_{t-s}p(X_s)-p(X_s)\right) = M^p_s $$ 
    which shows that $M^p$ is a martingale.
    
    We have
     \begin{align*}
     \E[|M_t^p|^2] &\leq 2\E[ p^2(X_t) ]+2\E\left[\left(\int_0^t\mathcal Gp(X_r)dr\right)^2\right]<\infty.
     \end{align*}
\end{proof}

In the previous lemma we made use of a square argument to ensure that $M^p$ is a martingale with finite second moment. When dealing with classical $d$-variable polynomials this can be improved because any $d$-variable polynomial of degree at most $n$ where $n$ is even can be written as a finite linearcombination of positive polynomials of degree $n$. This, however might fail for abstract polynomials or the infinite dimensional case. For this reason we require the moment condition $\mathcal{P}_2\cdot \mathcal{P}_2 \subset \mathcal{P}_4$ in the last lemma and that $pq$ is in the domain of $\mathcal G$ in the next theorem.

\begin{theorem}\label{t:L2-structure}
   Let $p,q\in\mathcal C\cap \mathcal L$ with $pq\in\mathcal C$
   and define
    \begin{align*}
        M_t^p &:= p(X_t) - \int_0^t \mathcal Gp(X_r)dr\quad{\text{and}} \\
        M_t^q &:= q(X_t) - \int_0^t \mathcal Gq(X_r)dr 
   \end{align*}
   for any $t\geq0$ and we define $a_{p,q}:=\mathcal G(pq)-p\mathcal Gq-q\mathcal Gp$. Recall that Lemma \ref{l:Mp martingale} yields that $M^p$, $M^q$ are martingales with finite second moments.
    
    Then the predictable quadratic covariation of $M^p$ and $M^q$ is given by
    $$ \langle M^p,M^q\rangle_t = \int_0^t a_{p,q}(X_s) ds,\quad t\geq 0. $$
\end{theorem}
\begin{proof}
   Denote by $\<M^p,M^q\>$ its predictable quadratic covariation in the sense of \cite[Theorem I.4.2]{JS}. Then
    $$ M^p_tM^q_t - \<M^p,M^q\>_t,\quad t\geq 0$$
    is a martingale. 
     Since $pq\in\mathcal C$,
     $$ M^{pq}_t := (pq)(X_t) - \int_0^t \mathcal G(pq)(X_r)dr $$
    is a martingale due to Lemma \ref{l:Mp martingale}. 
    We have 
    \begin{align*}
    \int_0^t \mathcal Gp(X_r)dr &\int_0^t \mathcal Gq(X_r)dr\\
    &= \int_0^t \left(\int_0^s\mathcal Gq(X_r)dr\mathcal Gp(X_s)+\int_0^s \mathcal Gp(X_r)dr\mathcal Gq(X_s)\right) ds, 
    \end{align*}
    and
    \begin{align*}
    &\int_0^t \mathcal Gp(X_r)dr q(X_t)\\ &\quad= \int_0^t q(X_r)\mathcal Gp(X_r)dr + \int_0^t \int_0^s \mathcal Gp(X_r)dr\mathcal Gq(X_s)ds + \int_0^t \int_0^s \mathcal Gp(X_r)dr dM^q_s.
    \end{align*}
    Hence, we find from expanding $M^pM^q$ and the identities above that $N$ given by
    \begin{align*}
        N_t &:= M^p_tM^q_t - \int_0^t a_{p,q}(X_r) dr \\
            &= M_t^{pq} - \int_0^t \int_0^r \mathcal Gp(X_s)ds dM_r^q - \int_0^t \int_0^r \mathcal Gq(X_s)ds dM_r^p
        \end{align*}
     for $t\geq 0$ is a $\sigma$-martingale in the sense of \cite[Definition III.6.33]{JS} due to \cite[Proposition III.6.42]{JS}. Consequently, $N-(M^pM^q - \<M^p,M^q\>)$ is a $\sigma$-martingale and since
      \begin{align}\label{eq:bracketminusintegral}
        N_t-(M^p_tM^q_t - \<M^p,M^q\>_t) = \<M^p,M^q\>_t-\int_0^ta_{p,q}(X_r)dr,\quad t\geq 0 
      \end{align}      it is predictable and of finite variation. Thus, $N-(M^pM^q - \<M^p,M^q\>)$ is a special semimartingale in the sense of \cite[Definition III.4.21(b)]{JS} and, hence, \cite[Proposition I.6.35]{JS} yields that it is a local martingale. \cite[Corollary I.3.16]{JS} yields that $N-(M^pM^q - \<M^p,M^q\>)$ is, in fact, constant $0$. Equation \eqref{eq:bracketminusintegral} yields the claim.
\end{proof}

Restricting our attention to diffusions in separable Hilbert spaces, we find the following 
corollary to our main result:
\begin{corollary}\label{c:stetig}
  Let $E$ be a separable Hilbert space with inner product $(\cdot,\cdot)_E$, $W$ an $E$-valued Brownian motion with covariance operator $Q$ in the sense of \cite[Sects.  3.5 and 4.4]{peszat.zabczyk.07} and assume that there is a linear subspace $\mathcal E$ of the continuous linear operators from $E$ to $\mathbb R$ which is contained in $\mathcal L\cap\mathcal C$ and such that $p^2\in\mathcal C$ for any $p\in\mathcal E$. We assume additionally that
    $$ X_t = X_0 + \int_0^t \beta_s ds + \int_0^t \sigma_s dW_s,\quad t\geq 0 $$
    where $\beta$ is a progressively measurable process, $E$-valued and locally integrable and $\sigma$ is a progressively measurable process, $L(E)$-valued and locally square integrable. 
    
    We denote the Lebesgue measure restricted to $\mathbb R_+$ by $\lambda$. Then
     \begin{align*}
         p(\beta_t) &= \mathcal Gp(X_t), \\
         ( \sigma_tQ\sigma^*_t p^*, q^*)_E &= a_{p,q}(X_t)
     \end{align*}
     $P\otimes \lambda$-a.s.\ for any $p,q\in\mathcal E$ where we have that $pq\in\mathcal C$ and $a_{p,q}$ is defined in Theorem \ref{t:L2-structure} and $p^*$ (resp.\ $q^*$) is the unique element in $E$ such that $p=(\cdot,p^*)_E$ (resp.\ $q=(\cdot,q^*)_E$).
     
    Moreover, if $\mathcal E$ is dense in the set of continuous linear operators, then $\mathcal G|_{\mathcal P_2}$ determines the drift $\beta$ and the covariance $\sigma Q\sigma^*$.
\end{corollary}
\begin{proof}
  Theorem \ref{t:affine drift} yields that 
   $$ p(X_t) = M_t^p + \int_0^t \mathcal Gp(X_s) ds,\quad t\geq 0 $$
   where $M^p$ is a martingale. Since
    $$ p(X_t) = \left( p(X_0) + \int_0^t p\sigma_sdW_s\right) + \int_0^t p(\beta_s) ds $$
    we find the claim for $\beta$ and see that
     $$ M_t^p = p(X_0) + \int_0^t p\sigma_sdW_s,\quad t\geq 0. $$
   Since $p,q\in\mathcal E$ and $p+q\in\mathcal E$ we find that 
    $$ pq = \frac12((p+q)^2-p^2-q^2)\in\mathcal C $$
   Thus $\<M^p,M^q\>_t = \int_0^t a_{p,q}(X_s)ds$ by Theorem \ref{t:L2-structure}. On the other hand we have
    $$ \<M^p,M^q\>_t = \int_0^t (\sigma_sQ\sigma^*_s p^*, q^*)_E ds $$
    and the claim follows.
   \end{proof}


\end{document}